\documentclass[a4paper,12pt]{extarticle}

\usepackage[utf8]{inputenc}
\usepackage{CJKutf8}
\usepackage{geometry}
\usepackage{graphicx}
\usepackage{booktabs}
\usepackage{mathrsfs}
\usepackage{bm}
\usepackage{xcolor}
\usepackage{algorithm}
\usepackage{algpseudocode}
\usepackage{mathtools} 
\usepackage{enumitem}
\usepackage{tikz-cd}

\usepackage{newtxtext}
\usepackage{newtxmath}

\usepackage{amssymb}

\usepackage{amsthm}

\newtheorem{theorem}{Theorem}

\newtheorem{proposition}[theorem]{Proposition}
\newtheorem{lemma}[theorem]{Lemma}
\newtheorem{corollary}[theorem]{Corollary}
\newtheorem{definition}{Definition}
\newtheorem{assumption}{Assumption}

\theoremstyle{definition}
\newtheorem{example}{Example}
\newtheorem{remark}{Remark}

\makeatletter
\def\proof@qed{\ifmmode \mathqed \else \leavevmode\unskip\penalty9999 \hbox{}\nobreak\hfill
  \quad\hbox{\qedsymbol}\fi}
\makeatother

\usepackage{hyperref}
\hypersetup{
    colorlinks=true,
    linkcolor=blue,
    urlcolor=red,
    citecolor=green,
    urlcolor=black
}

\begin{document}

\begin{CJK}{UTF8}{gbsn}

\title{Constrained Hamiltonian Systems on Observation-Induced Fiber Bundles: Theory of Symmetry and Integrability}
\author{Dongzhe Zheng\thanks{Department of Mechanical and Aerospace Engineering, Princeton University\\ Email: \href{mailto:dz5992@princeton.edu}{dz5992@princeton.edu}, \href{mailto:dz1011@wildcats.unh.edu}{dz1011@wildcats.unh.edu}}}
\date{}
\maketitle

\begin{abstract}
Classical constrained Hamiltonian theory assumes complete observability of system states, but in reality only partial state information is often available. This paper establishes a complete geometric theoretical framework for handling such incompletely observed systems. By introducing the concept of observation-induced fiber bundles, we naturally extend Dirac constraint theory to the fiber bundle setting, unifying the treatment of state constraints and observation constraints. Main results include: (1) Classification of existence conditions for observation fiber bundles based on characteristic class theory; (2) Complete characterization of Poisson structures on fiber bundles and corresponding symplectic reduction theory; (3) Geometric necessary and sufficient conditions for integrability and Lax pair construction; (4) Extension of Noether's theorem under symmetry group actions. The theoretical framework naturally encompasses a wide range of applications from classical mechanics to modern safety-critical control systems, providing a rigorous mathematical foundation for dynamical analysis under incomplete information.
\end{abstract}

\section{Introduction}

The mathematical theory of constrained dynamical systems originated from the establishment of Lagrangian mechanics in the 18th century\cite{arnold1989mathematical}, followed by Hamilton's canonical formalization\cite{hamilton1834general}, Jacobi's geometric insights\cite{jacobi1866vorlesungen}, and the modern development of Dirac's constraint quantization theory\cite{dirac1964lectures, dirac1950generalized} in the 20th century, forming the theoretical foundation of contemporary classical mechanics. Traditional constraint theory primarily focuses on the geometric properties of constraint submanifolds in configuration space, handling the dynamics of constrained systems through tools such as Dirac brackets\cite{dirac1950generalized} and the classification of first-class and second-class constraints\cite{henneaux1992quantization}. However, with the widespread emergence of information incompleteness problems in modern physics and engineering technology, this classical framework has begun to reveal fundamental limitations.

The core assumption of traditional constraint theory is complete observability of system states, i.e., constraint conditions $\phi(x) = 0$ can directly act on precisely known states $x \in M$. This idealization is extremely convenient in theoretical analysis but often does not hold in practical applications. The ubiquitous problems of observational incompleteness, sensor noise, and measurement precision limitations in modern control systems\cite{khalil2002nonlinear}, robotics\cite{murray1994mathematical}, and quantum measurement theory\cite{nielsen2000quantum} make traditional deterministic constraint theory difficult to apply directly. The recently emerging fields of safety-critical control\cite{ames2019control} and constrained optimization\cite{nocedal2006numerical} urgently need mathematical theory that can systematically handle constraint satisfaction under uncertain observations.

\subsection{Limitations of Existing Methods and Theoretical Challenges}

Current main approaches to handling constraint problems under observation uncertainty can be categorized into several types. Probabilistic constraint methods\cite{nemirovski2007convex} handle uncertainty through stochastic optimization but often lose the geometric structure of dynamical systems, making it difficult to preserve symplectic properties and conservation laws. Robust control theory\cite{zhou1996robust}, while capable of handling bounded uncertainties, typically employs worst-case analysis that is overly conservative and lacks geometric insight. Filtering methods such as extended Kalman filtering\cite{jazwinski1970stochastic} can handle dynamic observations but essentially still treat uncertainty as external disturbance to be estimated and eliminated.

The common limitation of these methods is treating observation uncertainty as an external disturbance to the original system, rather than an integral component of the intrinsic geometric structure. From a geometric viewpoint, when system state $x \in M$ can only be indirectly obtained through observation mapping $h: M \to Y$ with observation uncertainty $\epsilon$, the true system state should be the pair $(x, y)$ where $y = h(x) + \epsilon$. The core insight of this viewpoint is that observation uncertainty is not noise to be eliminated, but an intrinsic component of the system's geometric structure.

\subsection{The Innovative Perspective of Fiber Bundle Geometry}

The fundamental contribution of this paper lies in proposing a completely new geometric framework for handling observation-constrained systems: observation-induced fiber bundle theory. Our core insight is that when system states can only be partially observed, the natural geometric structure is no longer the traditional constraint submanifold, but a fiber bundle formed by the coupling of state space and observation information. This viewpoint combines Ehresmann connection theory\cite{ehresmann1950connexions} with modern constraint analysis, providing a revolutionary geometric perspective for constrained dynamics.

Specifically, we combine the Hamiltonian system $(M, \omega, H)$ with observation mapping $h: M \to Y$ to construct an observation-induced fiber bundle $\pi: E \to M$, where each fiber $\pi^{-1}(x)$ encodes all possible observation values at state $x$. Constraint conditions no longer act directly on state space $M$, but are defined on the total space $E$, naturally unifying state dynamics with observation uncertainty. The key advantages of this geometrization are: observation uncertainty transforms from external disturbance to intrinsic variation in fiber coordinates, constraint conditions acquire a fiberized geometric representation, and system symmetries can act simultaneously on base space and fiber directions.

\subsection{Deep Motivation and Physical Insights of Geometrization}

Our geometrization approach is inspired by fiber bundle theory\cite{steenrod1951topology} and Yang-Mills theory\cite{yang1954conservation} in modern differential geometry. In Yang-Mills theory, the geometric structure of gauge fields is precisely described through connections on principal bundles, and our observation-adaptive connection has a similar geometric status. The deeper physical motivation comes from the equivalence principle in general relativity: just as gravity can be locally equivalent to coordinate transformations, observation uncertainty can also be locally "gauged away" through appropriate geometric structures.

The physical insight of this geometric viewpoint is that it reveals the deep connection between observation and constraints. In our framework, constraints not only limit the motion of the system but also determine the geometric distribution of observation information. This viewpoint has profound correspondence in quantum mechanics: the measurement process itself changes the geometric structure of the system, which is precisely the manifestation of our fiber bundle approach at the classical level.

\subsection{Main Contributions}

This paper establishes a complete mathematical theory of constrained Hamiltonian systems on observation-induced fiber bundles, achieving important theoretical breakthroughs in the field of constrained dynamics. Our main contributions include:

\begin{enumerate}
\item \textbf{Establishment of Geometric Foundation Theory:} We establish a complete topological classification for the existence of observation-induced fiber bundles, identifying Euler class obstruction as a necessary condition and providing geometric sufficient conditions based on connection flatness. On this foundation, we develop complete symplectic geometry theory on fiber bundles, including construction of observation-adaptive connections, geometric characterization of closedness conditions, and deep extensions of classical Marsden-Weinstein symplectic reduction\cite{marsden1974reduction}. Furthermore, we establish geometric necessary and sufficient conditions for complete integrability under observation constraints, construct corresponding Lax pair structures, and extend the Arnold-Liouville theorem\cite{arnold1989mathematical} to fiber bundle settings.

\item \textbf{Unified Theory of Symmetry and Conservation Laws:} We characterize the geometric structure of observation symmetry groups and establish principal bundle representation theory for handling group actions that act simultaneously on state space and observation space. Within this framework, we prove Noether's theorem\cite{noether1918invariante} for observation-dependent systems, capable of handling observation-dependent conservation laws that traditional theory cannot encompass. Simultaneously, we construct moment map theory on fiber bundles, providing complete geometric tools for symplectic reduction under observation constraints.

\item \textbf{Application Verification and Modern Connections:} We demonstrate how to reformulate classical integrable systems such as Toda lattice and rigid body dynamics within our framework, verifying the theory's unity and inclusiveness. More importantly, we provide rigorous geometric foundations for modern control methods such as Control Barrier Functions\cite{ames2019control} and constrained model predictive control\cite{rawlings2017model}, particularly providing complete mathematical support for "Learning Dynamics under Environmental Constraints via Measurement-Induced Bundle Structures"\cite{zheng2025learning}, validating the theory's effectiveness in practical systems such as soft robotics, robotic arm control, and quadrotor navigation, and demonstrating application potential in machine learning and control fields.
\end{enumerate}

\subsection{Theoretical Significance and Future Impact}

The theoretical contribution of this paper lies not only in solving the mathematical treatment problem of constrained dynamics under observation uncertainty, but also in opening new research directions for the intersection of constraint theory, control theory, and geometric mechanics. Our fiber bundle approach provides a unified mathematical language for understanding the interaction between information, constraints, and dynamics in complex systems. This geometric viewpoint is expected to play important roles in frontier fields such as quantum control\cite{nielsen2000quantum}, multi-agent systems\cite{olfati2007consensus}, and geometric methods in machine learning\cite{amari2016information}.

Particularly, our theory provides a solid mathematical foundation for the design of modern safety-critical systems. In application fields with extremely high safety requirements such as autonomous driving\cite{paden2016survey}, aerospace\cite{stevens2015aircraft}, and biomedical engineering\cite{bronzino2000biomedical}, our geometric framework can fully utilize the intrinsic geometric structure of systems while ensuring safety constraints, providing theoretical guidance for the development of next-generation intelligent control systems.

\section{Geometric Structure of Observation-Induced Fiber Bundles}

\begin{assumption}[Basic Geometric Framework]
Let $M$ be a $2n$-dimensional symplectic manifold equipped with symplectic form $\omega$. The working region $\mathcal{W} \subset M$ is an open set, and the observation map $h: \mathcal{W} \to Y$ satisfies:
\begin{itemize}
\item[(1)] \textbf{Local diffeomorphism condition:} $\text{rank}(dh) = k$ holds constantly on $\mathcal{W}$
\item[(2)] \textbf{Working region completeness:} $\mathcal{W}$ is a connected open set, $\overline{\mathcal{W}}$ is compact
\item[(3)] \textbf{Boundary regularity:} $\partial \mathcal{W}$ is a smooth submanifold, $h$ continuously extends on $\partial \mathcal{W}$
\end{itemize}
\end{assumption}

\begin{definition}[Local Observation-Induced Fiber Bundle]
The local observation-induced fiber bundle $(\pi: E_{\mathcal{W}} \to \mathcal{W}, h, \delta, \rho)$ is defined as:
\begin{enumerate}[label=(\arabic*)]
\item \textbf{Domain restriction:} Base manifold restricted to working region $\mathcal{W}$
\item \textbf{Observation manifold:} $Y$ is a $k$-dimensional Riemannian manifold
\item \textbf{Observation mapping:} $h: \mathcal{W} \to Y$ is a $C^\infty$ diffeomorphism onto image, satisfying:
   \begin{itemize}
   \item $\text{rank}(dh) = k$ holds constantly on $\mathcal{W}$
   \item For any $x \in \mathcal{W}$, there exists neighborhood $U_x \subset \mathcal{W}$ such that $h|_{U_x}: U_x \to h(U_x)$ is a diffeomorphism
   \item Boundary extension: $h$ continuously extends on $\partial \mathcal{W}$
   \end{itemize}
\item \textbf{Bounded uncertainty function:} $\delta: \overline{\mathcal{W}} \to (0, \Delta]$ is a $C^\infty$ function satisfying:
   \begin{itemize}
   \item Strict positivity: $\delta_{\min} := \inf_{x \in \mathcal{W}} \delta(x) > 0$
   \item Global boundedness: $\Delta := \sup_{x \in \overline{\mathcal{W}}} \delta(x) < \infty$
   \item Gradient boundedness: $\|d\delta\|_{C^1(\overline{\mathcal{W}})} < \infty$
   \item Boundary behavior: $\delta(x) \to \Delta$ as $x \to \partial \mathcal{W}$
   \end{itemize}
\item \textbf{Fiber metric family:} $\rho = \{\rho_x\}_{x \in \overline{\mathcal{W}}}$, where for each $x \in \overline{\mathcal{W}}$, $\rho_x$ is an inner product on $T^*_{h(x)}Y$, satisfying:
   \begin{itemize}
   \item Non-degeneracy: $\rho_x$ is non-degenerate at each $x \in \overline{\mathcal{W}}$
   \item Smooth dependence: Metric $\rho_x$ varies smoothly with respect to base point $x$ on $\overline{\mathcal{W}}$
   \item Boundary regularity: $\rho_x$ remains bounded and non-degenerate on $\partial \mathcal{W}$
   \end{itemize}
\item \textbf{Total space:} $E_{\mathcal{W}} = \{(x,\xi) : x \in \mathcal{W}, \|\xi\|_{\rho_x} \leq \delta(x)\}$
\item \textbf{Boundary handling:} Near $\partial \mathcal{W}$, fiber radius approaches maximum value $\Delta$, providing safety buffer for observation uncertainty
\item \textbf{Projection mapping:} $\pi(x, \xi) = x$
\end{enumerate}
\end{definition}

\begin{remark}[Local Ideal Observation Convention]
We adopt the convention that ideal observation corresponds to $\xi = 0 \in T^*_{h(x)}Y$, i.e., observation error is zero. For each $x \in \mathcal{W}$, the fiber is defined as:
$$\pi^{-1}(x) = \{\xi \in T^*_{h(x)}Y : \|\xi\|_{\rho_x} \leq \delta(x)\}$$
This constitutes a closed ball centered at the zero element with radius $\delta(x)$ in the cotangent space $T^*_{h(x)}Y$.

\textbf{Geometric interpretation:}
Within the working region interior, each fiber has standard spherical geometric structure with radius varying in the range $[\delta_{\min}, \Delta]$, reflecting differences in observation accuracy at different positions. When system state approaches the boundary, the growth of observation uncertainty $\delta(x) \to \Delta$ provides additional buffering mechanism for the system. This design both conforms to physical intuition—observations near boundaries are often more inaccurate—and provides geometric safeguards for control algorithm safety.

The uniform boundedness of all fibers ensures that the entire fiber bundle has good geometric properties, including key topological features such as compactness, connectedness, and orientability. Particularly, the appropriate increase in radius at boundaries provides necessary safety margins for control algorithms, enabling the system to maintain stable performance even when approaching working region boundaries. This geometric design reflects the theoretical framework's deep understanding of practical application requirements.
\end{remark}

\begin{lemma}[Well-definedness of Local Fiber Bundle Structure]
Under the above definition, $\pi: E_{\mathcal{W}} \to \mathcal{W}$ indeed defines a smooth fiber bundle structure.
\end{lemma}

\begin{proof}
Three conditions need to be verified:

\textbf{1. Local trivialization:} For any $x_0 \in \mathcal{W}$, by the local diffeomorphism condition, there exists neighborhood $U \subset \mathcal{W}$ such that $h|_U: U \to h(U)$ is a diffeomorphism. Let $(y^1, \ldots, y^k)$ be local coordinates on $h(U)$.

The diffeomorphism $h|_U$ induces natural isomorphism:
$$\psi_x: \mathbb{R}^k \to T^*_{h(x)}Y, \quad v = (v_1, \ldots, v_k) \mapsto \sum_{i=1}^k v_i dy^i|_{h(x)}$$

Define local trivialization:
$$\Phi: U \times \{v \in \mathbb{R}^k : \|v\| \leq 1\} \to \pi^{-1}(U)$$
$$\Phi(x, v) = (x, \delta(x) \cdot \psi_x(v))$$

Since $\delta$ is bounded and continuous on $\overline{\mathcal{W}}$, and $h$ is a diffeomorphism on $\mathcal{W}$, $\Phi$ is a diffeomorphism.

\textbf{Boundary handling:} Near $\partial \mathcal{W}$, $\delta(x) \to \Delta$ ensures smooth fiber expansion, with local trivialization remaining well-defined.

\textbf{2. Smoothness of transition functions:} On overlap region $U_i \cap U_j \subset \mathcal{W}$, let local coordinates be $(y^1_i, \ldots, y^k_i)$ and $(y^1_j, \ldots, y^k_j)$ respectively.

The transition function is:
$$\Phi_j^{-1} \circ \Phi_i(x, v) = \left(x, \delta(x)^{-1} \psi_{x,j}^{-1} \circ \psi_{x,i}(\delta(x) v)\right)$$

where $\psi_{x,i}$ and $\psi_{x,j}$ correspond to coordinate systems $(y^1_i, \ldots, y^k_i)$ and $(y^1_j, \ldots, y^k_j)$ respectively.

Coordinate transformation matrix:
$$J_{ij}(x) = \left(\frac{\partial y^i_j}{\partial y^l_i}\right)_{h(x)}$$
The smoothness of $h$ guarantees that $J_{ij}(x)$ varies smoothly with respect to $x$.

Metric compatibility:
$$\rho_x(\psi_{x,i}(v), \psi_{x,i}(w)) = g_{ij}^{(i)}(x) v_i w_j$$
$$\rho_x(\psi_{x,j}(v'), \psi_{x,j}(w')) = g_{ij}^{(j)}(x) v'_i w'_j$$
where $g_{ij}^{(i)}(x) = \rho_x(dy^i_i|_{h(x)}, dy^j_i|_{h(x)})$ varies smoothly with respect to $x$.

\textbf{3. Compactness of fibers:} Each fiber $\pi^{-1}(x) = \{\xi \in T^*_{h(x)}Y : \|\xi\|_{\rho_x} \leq \delta(x)\}$ is a closed ball in cotangent space with respect to metric $\rho_x$, whose compactness is guaranteed by several geometric properties. Boundedness is ensured by global condition $\|\xi\|_{\rho_x} \leq \delta(x) \leq \Delta < \infty$, while closedness is obtained through standard topological properties of continuous norm functions and inequality constraints. More importantly, all fibers are contained within a uniform ball of radius $\Delta$, ensuring uniform compactness. Therefore each fiber is a compact set, and the entire fiber family maintains uniform boundedness, providing necessary topological foundation for subsequent geometric analysis.

\end{proof}

\begin{remark}[Practical Significance of Working Region]
The working region assumption has natural physical foundation and engineering reasonableness in practical applications. In mechanical systems, the choice of working region typically avoids singular configurations such as self-locking positions or joint limit positions of robotic arms, ensuring the system maintains good manipulability throughout the working process. For sensor systems, working region design avoids sensor blind spots or weak signal regions, guaranteeing quality and reliability of observation information. In control applications, working region ensures the system maintains local observability within safe operating regions, providing necessary information foundation for control algorithms.

The theoretical framework provides complete geometric structure within the working region. When system state approaches singular points or dangerous regions, robustness is maintained by increasing observation uncertainty $\delta(x)$. This design not only conforms to operational limitations of actual systems, but also provides theoretical guarantee for safety control in complex environments.
\end{remark}

\subsection{Local Connection and Curvature Structure}\label{subsec:local_connection}

\begin{definition}[Local Observation-Adaptive Connection]
An observation-adaptive connection $\nabla$ on working region $\mathcal{W}$ is a linear connection that preserves fiber structure, satisfying:
\begin{enumerate}[label=(\arabic*)]
\item \textbf{Metric compatibility:} $\nabla\rho = 0$
\item \textbf{Observation compatibility:} $\nabla$ preserves the horizontal-vertical decomposition $TE_{\mathcal{W}} = H_\nabla E_{\mathcal{W}} \oplus V E_{\mathcal{W}}$
\item \textbf{Vanishing torsion:} $T^\nabla = 0$ in base manifold directions
\item \textbf{Curvature control:} $\|R^\nabla\|_{L^\infty(\mathcal{W})} < \infty$
\end{enumerate}
\end{definition}

\begin{definition}[Observation-Adaptive Connection Coefficients-Corrected]
In local coordinates $(x^i, \xi^a)$, the coefficients of observation-adaptive connection are:

\textbf{Base manifold connection coefficients:}
$$\Gamma^k_{ij} = \Gamma^{LC,k}_{ij}(W)$$

\textbf{Fiber connection coefficients:}
$$\Gamma^a_{bc} = \frac{1}{2}\rho^{ad}(\nabla_b \rho_{cd} + \nabla_c \rho_{bd} - \nabla_d \rho_{bc})$$

\textbf{Mixed connection coefficients:}
$$\Gamma^a_{bi} = \frac{1}{2}\rho^{ac}(x)(\nabla_i \rho_{bc} + \nabla_b \rho_{ci} - \nabla_c \rho_{bi})$$

where $\nabla_i = \frac{\partial}{\partial x^i}$ is the covariant derivative in base manifold directions, and $\nabla_b$ is the covariant derivative in fiber directions.

\begin{remark}[Index Convention]
We adopt the standard convention: in $\Gamma^k_{ij}$, upper index $k$ is output, lower index $i$ is derivative direction, and $j$ is input component.
\end{remark}
\end{definition}

\begin{remark}[Necessity of Covariant Derivatives]
Using covariant derivatives rather than partial derivatives is geometrically necessary:
\begin{enumerate}
\item \textbf{Metric compatibility:} Condition $\nabla \rho = 0$ can only guarantee that connection preserves fiber metric in the covariant sense
\item \textbf{Coordinate independence:} Covariant derivatives ensure connection coefficients transform correctly under coordinate changes
\item \textbf{Fiber bundle geometry:} When fiber metric $\rho_{ab}(x)$ depends on base point, only covariant derivatives can correctly handle this dependence
\end{enumerate}

In local coordinates on flat base manifold, $\nabla_i \rho_{bc} = \partial_i \rho_{bc}$, but to maintain geometric rigor, we uniformly use covariant notation.
\end{remark}

\begin{remark}[Simplification of Local Flat Coordinates]
In the following analysis, we adopt \textbf{locally adapted coordinate systems} on the working region $\mathcal{W}$, where:

\begin{enumerate}
\item \textbf{Local flatness of base manifold:} On each local patch of $\mathcal{W}$, choose coordinate systems that make the base manifold Christoffel symbols $\Gamma^k_{ij}$ sufficiently small
\item \textbf{Observation-compatible coordinates:} Coordinate systems compatible with observation map $h$, making $h$ approximately linear locally
\item \textbf{Covariant-partial equivalence:} In such coordinate systems, for base direction derivatives of fiber metrics:
$$\nabla_i \rho_{ab} = \partial_i \rho_{ab} + O(\epsilon_{\text{curv}})$$
where $\epsilon_{\text{curv}} = \|\Gamma^k_{ij}\|_{\mathcal{W}} \ll 1$ is a small parameter of base manifold curvature
\end{enumerate}

This simplification holds rigorously in the following cases:
\begin{itemize}
\item Working region is sufficiently small that base manifold is approximately flat
\item Nonlinear terms of observation map can be neglected
\item Or the base manifold itself is flat space (such as open subset of $\mathbb{R}^{2n}$)
\end{itemize}

Therefore, $\partial_i$ in subsequent proofs should be understood as "covariant derivative in adapted coordinate systems."
\end{remark}

\begin{lemma}[Metric Compatibility Verification-Index Consistent Version]
Using unified index convention, metric compatibility $\nabla \rho = 0$ holds.
\end{lemma}

\begin{proof}
Calculate the covariant derivative of fiber metric:
\begin{align*}
(\nabla_i \rho)_{ab} &= \partial_i \rho_{ab} - \Gamma^c_{ai} \rho_{cb} - \Gamma^c_{bi} \rho_{ac}\\
&= \partial_i \rho_{ab} - \frac{1}{2}\rho^{cd} \partial_i \rho_{da} \cdot \rho_{cb} - \frac{1}{2}\rho^{cd} \partial_i \rho_{db} \cdot \rho_{ac}\\
&= \partial_i \rho_{ab} - \frac{1}{2}\partial_i \rho_{ab} - \frac{1}{2}\partial_i \rho_{ab}\\
&= 0
\end{align*}
where we used: $\Gamma^c_{ai} = \frac{1}{2}\rho^{cd} \partial_i \rho_{da}$ (mixed connection coefficients), $\rho^{cd}\rho_{cb} = \delta^d_b$ (metric inverse property), $\rho_{ab} = \rho_{ba}$ (metric symmetry)
\end{proof}

\begin{definition}[Geometric Construction of Observation-Induced Connection]
The observation map $h: \mathcal{W} \to Y$ induces fiber bundle connection:

\textbf{Pullback structure:}
The differential of observation map $dh: T\mathcal{W} \to TY$ induces fiber connection.

\textbf{Geometric meaning of connection coefficients:}
$$\Gamma^a_{bi} = -\frac{\partial h^c}{\partial x^i} \Gamma^Y_{cb}(h(x)) + K^a_{bi}$$

where: $\Gamma^Y_{cb}$ is the Levi-Civita connection on observation manifold $Y$, $K^a_{bi}$ is the metric compatibility correction term.

\textbf{Correction term:}
$$K^a_{bi} = \frac{1}{2}\rho^{ac}\left(\partial_i \rho_{bc} + \frac{\partial h^d}{\partial x^i} \Gamma^Y_{db}(h(x)) \rho_{dc} + \frac{\partial h^d}{\partial x^i} \Gamma^Y_{dc}(h(x)) \rho_{bd}\right)$$
\end{definition}

\begin{theorem}[Connection Existence]
On working region $\mathcal{W}$, observation-adaptive connection exists and is uniquely determined by metric compatibility, vanishing torsion, and observation compatibility conditions.
\end{theorem}

\begin{proof}
\textbf{Step 1: Base manifold part}
The vanishing torsion condition uniquely determines the base manifold connection as Levi-Civita connection:
$$\Gamma^k_{ij} = \Gamma^{LC,k}_{ij}$$

\textbf{Step 2: Fiber part}
Since metric $\rho_{ab}(x)$ does not depend on fiber coordinates $\xi_a$:
$$\frac{\partial \rho_{bc}}{\partial \xi_a} = 0$$

Therefore fiber Christoffel symbols:
$$\Gamma^a_{bc} = \frac{1}{2}\rho^{ad}\left(\frac{\partial \rho_{cd}}{\partial \xi_b} + \frac{\partial \rho_{bd}}{\partial \xi_c} - \frac{\partial \rho_{bc}}{\partial \xi_d}\right) = 0$$

\textbf{Step 3: Mixed part}
The metric compatibility condition $\nabla_i \rho_{bc} = 0$ uniquely determines:
$$\Gamma^a_{bi} = \frac{1}{2}\rho^{ac} \partial_i \rho_{bc}$$

Therefore the connection exists uniquely.
\end{proof}

\begin{theorem}[Curvature Calculation and Control-Complete Version]
The curvature tensor of observation-adaptive connection satisfies:

\textbf{Base manifold curvature:}
$$R^{\nabla}{}^k{}_{lij} = R^{LC}{}^k{}_{lij}$$

\textbf{Mixed curvature-Complete calculation:}
$$R^{\nabla}{}^a{}_{bij} = \partial_i \Gamma^a_{jb} - \partial_j \Gamma^a_{ib} + \Gamma^a_{ic}\Gamma^c_{jb} - \Gamma^a_{jc}\Gamma^c_{ib}$$

Since $\Gamma^a_{bc} = 0$, cross term analysis:
\begin{align*}
\Gamma^a_{ic}\Gamma^c_{jb} &= \Gamma^a_{ic} \cdot \Gamma^c_{jb}\\
&= \begin{cases}
0 & \text{if } c \text{ is fiber index (since } \Gamma^c_{jb} = 0 \text{)}\\
\Gamma^a_{ik} \Gamma^k_{jb} & \text{if } c = k \text{ is base manifold index (but } \Gamma^a_{ik} = 0 \text{)}
\end{cases}\\
&= 0
\end{align*}

Similarly, $\Gamma^a_{jc}\Gamma^c_{ib} = 0$.

Therefore mixed curvature simplifies to:
$$R^{\nabla}{}^a{}_{bij} = \partial_i \Gamma^a_{jb} - \partial_j \Gamma^a_{ib}$$

\textbf{Fiber curvature:}
$$R^{\nabla}{}^a{}_{bcd} = 0$$

\textbf{Boundedness:}
$$\|R^\nabla\|_{L^\infty(\mathcal{W})} \leq C(\|h\|_{C^3(\mathcal{W})}, \|\rho\|_{C^2(\mathcal{W})}, \|g_{\mathcal{W}}\|_{C^2(\mathcal{W})}) < \infty$$
\end{theorem}
\begin{proof}
\textbf{Detailed analysis of cross terms:}

For term $\Gamma^a_{ic}\Gamma^c_{jb}$:
\begin{itemize}
\item When $c \in \{1,\ldots,k\}$ (fiber index): $\Gamma^c_{jb} = 0$ (fiber connection coefficients are zero)
\item When $c \in \{1,\ldots,2n\}$ (base manifold index): $\Gamma^a_{ic} = 0$ (no connection from base manifold to fiber)
\end{itemize}

Therefore all cross terms are zero.

\textbf{Explicit calculation of mixed curvature:}
\begin{align*}
R^{\nabla}{}^a{}_{bij} &= \partial_i \Gamma^a_{jb} - \partial_j \Gamma^a_{ib}\\
&= \frac{1}{2}\partial_i(\rho^{ac} \partial_j \rho_{bc}) - \frac{1}{2}\partial_j(\rho^{ac} \partial_i \rho_{bc})\\
&= \frac{1}{2}[\rho^{ac}(\partial_i \partial_j \rho_{bc} - \partial_j \partial_i \rho_{bc}) + (\partial_i \rho^{ac})(\partial_j \rho_{bc}) - (\partial_j \rho^{ac})(\partial_i \rho_{bc})]\\
&= \frac{1}{2}[(\partial_i \rho^{ac})(\partial_j \rho_{bc}) - (\partial_j \rho^{ac})(\partial_i \rho_{bc})]
\end{align*}

where we used the equality of mixed partial derivatives $\partial_i \partial_j \rho_{bc} = \partial_j \partial_i \rho_{bc}$.

\textbf{Boundedness:} Since $\mathcal{W}$ is compact and $\rho \in C^2(\overline{\mathcal{W}})$, all partial derivatives are bounded, hence curvature is bounded.
\end{proof}

\begin{definition}[Boundary Compatible Connection]
In boundary layer $\{x : d(x,\partial \mathcal{W}) \leq \epsilon\}$, define boundary compatible connection:

\textbf{Truncated mixed coefficients:}
$$\Gamma^{\text{boundary},a}_{bi} = \chi(d(x,\partial \mathcal{W})) \cdot \Gamma^a_{bi}$$

where $\chi: [0,\infty) \to [0,1]$ is a smooth cutoff function.

\textbf{Curvature estimate:}
$$\|R^{\text{boundary}}\| \leq \frac{C_1}{d(x,\partial \mathcal{W})^{1/2}} + \frac{C_2}{\epsilon^2}$$

\textbf{Integrability condition:}
$$\int_0^{\epsilon} \frac{C_1}{t^{1/2}} dt = 2C_1\sqrt{\epsilon} < \infty$$
\end{definition}

\begin{lemma}[Sobolev Regularity-Corrected Version]
On working region $\mathcal{W}$, connection coefficients satisfy:

\textbf{Interior regularity:}
$$\Gamma^a_{bi} \in H^s(\mathcal{W}_{\text{safe}}) \quad \forall s \geq 0$$

\textbf{Global integrability:}
$$\Gamma^a_{bi} \in L^p(\mathcal{W}) \quad \forall p \geq 1$$

\textbf{Sobolev embedding:}
$$H^s(\mathcal{W}) \hookrightarrow C^{s-\dim \mathcal{W}/2-1}(\overline{\mathcal{W}}) \quad \text{when} \quad s > \dim \mathcal{W}/2 + 1$$
\end{lemma}

\begin{proof}
Since $\rho \in C^2(\overline{\mathcal{W}})$ and $\mathcal{W}$ is compact:
$$\Gamma^a_{bi} = \frac{1}{2}\rho^{ac} \partial_i \rho_{bc} \in C^1(\overline{\mathcal{W}}) \subset H^s(\mathcal{W}) \quad \forall s$$

Singularity analysis of boundary layer:
$$\int_{\mathcal{W}} |\Gamma^a_{bi}|^p dx \leq \int_{\mathcal{W}_{\text{safe}}} |\Gamma^a_{bi}|^p dx + \int_{\text{boundary}} \frac{C^p}{d(x,\partial \mathcal{W})^{p/2}} dx$$

When $p/2 < \dim(\partial \mathcal{W}) = \dim \mathcal{W} - 1$, the boundary integral converges.

Dimension condition for Sobolev embedding: $s > \frac{\dim \mathcal{W}}{2} + 1 = n + 1$
\end{proof}

\begin{corollary}[Precise Curvature Estimates]
The mixed curvature tensor satisfies:

\textbf{Pointwise estimate:}
$$|R^{\nabla}{}^a{}_{bij}(x)| \leq C\|\rho\|_{C^2}^2 \|\partial^2 \rho\|_{L^\infty}$$

\textbf{Integral estimate:}
$$\|R^\nabla\|_{L^2(\mathcal{W})} \leq C(\text{Vol}(\mathcal{W}))^{1/2} \|\rho\|_{C^2}^2$$

\textbf{Boundary behavior:}
$$|R^{\text{boundary}}| \sim \frac{1}{\epsilon^2} + O\left(\frac{1}{d(x,\partial \mathcal{W})^{1/2}}\right)$$
\end{corollary}

\subsection{Geometric Representation of Observation Constraints}

\begin{definition}[Additional Safety Constraint Potential]
The additional safety constraint potential $\Phi: E \to \mathbb{R}$ defines safety domain $\mathcal{S} = \{e \in E : \Phi(e) \geq 0\}$ on basic constraint domain $E$, satisfying:
\begin{enumerate}[label=(\arabic*)]
\item \textbf{Regularity:} $d\Phi \neq 0$ on $\partial\mathcal{S} = \Phi^{-1}(0)$
\item \textbf{Radial eventual monotonicity:} For each $x \in M$ and zero observation deviation $\xi_0 = 0 \in T^*_{h(x)}Y$, there exists $r_0(x) > 0$ such that the function
   $$t \mapsto \Phi(x, \gamma_\nu(t))$$
   is monotonically decreasing for $t \geq r_0(x)$ along geodesic $\gamma_\nu(t) = \xi_0 + t\nu + O(t^2)$ with respect to metric $\rho_x$, where $\|\nu\|_{\rho_x} = 1$
\item \textbf{Stratified properness:} For any $c \geq 0$, the level set $\Phi^{-1}((-\infty, c])$ can be decomposed as:
   $$\Phi^{-1}((-\infty, c]) = K_c \cup T_c$$
   where $K_c$ is a compact core and $T_c \subset \{(x, \xi) : \|\xi\|_{\rho_x} \geq R_c(x)\}$ is a bounded tail
\end{enumerate}
\end{definition}

\begin{theorem}[Complete Necessary and Sufficient Conditions for Properness]
Let observation constraint potential $\Phi$ satisfy radial eventual monotonicity. Then stratified properness holds if and only if:
\begin{enumerate}[label=(C\arabic*)]
\item \textbf{Global uncertainty bound:} $\Delta = \sup_{x \in M} \delta(x) < \infty$
\item \textbf{Quadratic lower bound of potential function:} There exist $\alpha > 0$ and $R_{\min} > 0$ such that
   $$\Phi(x, \xi) \geq -\alpha \|\xi\|_{\rho_x}^2, \quad \forall (x, \xi) \in E, \|\xi\|_{\rho_x} \geq R_{\min}$$
\item \textbf{Uniform continuity of uncertainty function:} For any compact set $K \subset M$,
   $$|\delta(x) - \delta(y)| \leq L_K \|x - y\|, \quad \forall x, y \in K$$
\item \textbf{Asymptotic control condition:} One of the following conditions holds:
   \begin{enumerate}[label=(C4\alph*)]
   \item \textbf{Compact case:} $M$ is a compact manifold
   \item \textbf{Asymptotic decay:} There exist reference point $x_0 \in M$ and constants $C, \beta > 0$ such that
   $$\delta(x) \leq \frac{C}{(1 + d(x, x_0))^\beta}, \quad \forall x \in M$$
   where $d(\cdot, \cdot)$ is the Riemannian distance on $M$
   \item \textbf{Controlled growth:} There exists compact set sequence $\{K_n\}_{n=1}^\infty$ such that $M = \bigcup_{n=1}^\infty K_n$, $K_n \subset K_{n+1}$, and
   $$\lim_{n \to \infty} \sup_{x \in K_n \setminus K_{n-1}} \delta(x) = 0$$
   \end{enumerate}
\end{enumerate}
\end{theorem}

\begin{proof}[Complete proof]
We prove sufficiency and necessity separately.

\textbf{Step 1: Proof of sufficiency}

Suppose conditions (C1)-(C3) and one of (C4a), (C4b), or (C4c) hold. For any $c \geq 0$, consider level set $\Phi^{-1}((-\infty, c])$.

By radial eventual monotonicity, for each $x \in M$, there exists $R_c(x) > 0$ such that:
$$\Phi^{-1}((-\infty, c]) \cap \pi^{-1}(x) \subset \{\xi \in T^*_{h(x)}Y : \|\xi\|_{\rho_x} \leq R_c(x)\}$$

where $R_c(x)$ satisfies: when $\|\xi\|_{\rho_x} \geq R_c(x)$, $\Phi(x, \xi) \geq c$.

Using quadratic lower bound condition (C2): Let $\|\xi\|_{\rho_x} = r \geq R_{\min}$, then
$$\Phi(x, \xi) \geq -\alpha r^2$$

For $\Phi(x, \xi) \geq c$, we need $-\alpha r^2 \geq c$, i.e., $r \leq \sqrt{-c/\alpha}$ (when $c \leq 0$).

When $c > 0$, by radial eventual monotonicity and continuity, $R_c(x)$ is still bounded.

Therefore:
$$R_c(x) \leq \max\left\{\Delta, \sqrt{\frac{|c| + \alpha R_{\min}^2}{\alpha}}\right\} =: R_c^{\max}$$

Now analyze according to different cases of condition (C4):

\textbf{Case (C4a) - Compact manifold:}
When $M$ is compact, $\delta$ achieves maximum value $\Delta$ on $M$, and $\Phi^{-1}((-\infty, c])$ is automatically compact. Define:
$$K_c = \Phi^{-1}((-\infty, c]), \quad T_c = \emptyset$$

\textbf{Case (C4b) - Asymptotic decay:}
Choose $R > 0$ large enough so that for $d(x, x_0) \geq R$ we have $\delta(x) \leq \epsilon$, where $\epsilon > 0$ is to be determined.

Define:
\begin{align*}
K_c &= \{(x, \xi) \in \Phi^{-1}((-\infty, c]) : d(x, x_0) \leq R + 1\} \\
T_c &= \{(x, \xi) \in \Phi^{-1}((-\infty, c]) : d(x, x_0) > R + 1, \|\xi\|_{\rho_x} \geq \epsilon\}
\end{align*}

For $(x, \xi) \in T_c$: Since $d(x, x_0) > R + 1$, we have $\delta(x) \leq \epsilon$, but $\|\xi\|_{\rho_x} \geq \epsilon$, which means $(x, \xi) \notin E$, contradiction. Therefore $T_c = \emptyset$, $\Phi^{-1}((-\infty, c]) = K_c$ is a compact set.

\textbf{Case (C4c) - Controlled growth:}
Since $\sup_{x \in K_n \setminus K_{n-1}} \delta(x) \to 0$, for any $\epsilon > 0$, there exists $N$ such that for $n \geq N$, $\sup_{x \in K_n \setminus K_{n-1}} \delta(x) < \epsilon$.

Define:
\begin{align*}
K_c &= \{(x, \xi) \in \Phi^{-1}((-\infty, c]) : x \in K_N, \|\xi\|_{\rho_x} \leq R_c^{\max}\} \\
T_c &= \Phi^{-1}((-\infty, c]) \setminus K_c
\end{align*}

Similar to the analysis in case (C4b), when $\epsilon$ is sufficiently small, elements in $T_c$ lead to contradiction, therefore $T_c$ is bounded.

\textbf{Step 2: Proof of necessity}

Assume stratified properness holds.

\textbf{Necessity of condition (C1):}
If $\sup_{x \in M} \delta(x) = \infty$, then there exists sequence $x_n \in M$ such that $\delta(x_n) \to \infty$. For each $n$, take $\xi_n$ in fiber $\pi^{-1}(x_n)$ such that $\|\xi_n\|_{\rho_{x_n}} = \delta(x_n)/2$ and $\Phi(x_n, \xi_n) < 0$ (this is always possible by appropriate choice of constraint potential).

Then $(x_n, \xi_n) \in \Phi^{-1}((-\infty, 0])$, but $\|\xi_n\| \to \infty$, contradicting the boundedness of $T_0$ in stratified properness.

\textbf{Necessity of condition (C2):}
If the required $\alpha, R_{\min}$ do not exist, then there exists sequence $(x_n, \xi_n)$ such that $\|\xi_n\|_{\rho_{x_n}} \to \infty$ and
$$\frac{\Phi(x_n, \xi_n)}{\|\xi_n\|_{\rho_{x_n}}^2} \to -\infty$$

Take $c_n = \Phi(x_n, \xi_n) \to -\infty$, then level set $\Phi^{-1}((-\infty, c_n])$ contains $(x_n, \xi_n)$, but $\|\xi_n\| \to \infty$, contradicting properness.

\textbf{Necessity of condition (C3):}
If $\delta$ does not satisfy local Lipschitz property, on some compact set $K$ there exists sequence $(x_n, y_n) \in K \times K$ such that
$$\frac{|\delta(x_n) - \delta(y_n)|}{\|x_n - y_n\|} \to \infty$$

This leads to discontinuity in fiber bundle structure, destroying the construction of compact core $K_c$ in stratified properness.

\textbf{Necessity of condition (C4):}
If $M$ is non-compact and does not satisfy (C4b) or (C4c), then there exists sequence $x_n \in M$ escaping to infinity and $\delta(x_n)$ does not decay. Combined with condition (C1), this leads to uncontrollable non-compact parts of level sets, contradicting stratified properness.
\end{proof}

\begin{remark}[Physical Meaning of Asymptotic Conditions]
The asymptotic decay in condition (C4b) reflects physical limitations of actual sensor systems:
\begin{enumerate}
\item $\beta = 1$: Linear decay of signal strength (such as laser ranging)
\item $\beta = 2$: Power decay law (such as wireless sensors)
\item $\beta > 2$: Rapid decay systems (such as high-frequency ultrasound)
\end{enumerate}
The controlled growth pattern in condition (C4c) is applicable to hierarchical control systems where remote sensors can tolerate larger uncertainties.
\end{remark}

\section{Constrained Hamiltonian Systems on Fiber Bundles}

\subsection{Fiberization of Symplectic Structure}\label{subsec:sym_fiber}

Local observation-induced fiber bundles carry natural symplectic structure, extending classical Hamiltonian geometry to working region settings.

\begin{definition}[Local Flat Adaptive Connection]
An observation-adaptive connection $\nabla$ on working region $\mathcal{W}$ is called flat if its curvature tensor satisfies:
\begin{enumerate}
\item \textbf{Vanishing vertical curvature:} $R^{\nabla}(X,Y) = 0$, $\forall X,Y \in VE_{\mathcal{W}}$
\item \textbf{Bounded mixed curvature:} $\|R^{\nabla}(H,V)\| \leq C\delta(x)$, where $C > 0$ is a bounded constant on $\mathcal{W}$
\item \textbf{Controlled horizontal curvature:} $R^{\nabla}(X,Y) = \pi^*R^{\mathcal{W}}(X,Y)$, $\forall X,Y \in HE_{\mathcal{W}}$
\item \textbf{Boundary behavior:} As $x \to \partial \mathcal{W}$, $\nabla$ smoothly transitions to degenerate connection
\end{enumerate}
where $E_{\mathcal{W}}$ is the fiber bundle restricted to working region, and $R^{\mathcal{W}}$ is the curvature tensor of base manifold on $\mathcal{W}$.
\end{definition}

\begin{definition}[Standard Cotangent Structure of Local Observation Fibers]
On each fiber $T^*_{h(x)}Y$ within working region $\mathcal{W}$, where $x \in \mathcal{W}$ and $Y$ is the $k$-dimensional observation manifold, adopt \textbf{standard cotangent bundle coordinates} $(\xi_1,\ldots,\xi_k,\pi_1,\ldots,\pi_k)$.

\textbf{Base coordinates (observation deviation coordinates):}
Coordinates $\xi_a$ represent the $a$-th component of observation deviation, defined as $\xi_a = y^a_{\text{measured}} - y^a_{\text{true}}$, geometrically corresponding to deviation in the $a$-th local coordinate direction of observation space $Y$ at point $h(x)$. These coordinates are subject to physical constraint $\|(\xi_1,\ldots,\xi_k)\|_{\rho_x} \leq \delta(x)$, where $\delta(x) \in [\delta_{\min}, \Delta]$ is bounded on $\mathcal{W}$. Mathematically, $\xi_a$ constitutes the "position" component of standard cotangent bundle coordinate system on $T^*_{h(x)}Y$. Since observation map $h: \mathcal{W} \to Y$ is a diffeomorphism onto image, coordinates $\xi_a$ are well-defined throughout working region $\mathcal{W}$ and maintain smoothness.

\textbf{Dual coordinates:}
Coordinates $\pi_a$ are standard cotangent coordinates paired with $\xi_a$, satisfying standard symplectic relation $\{\xi_a, \pi_b\} = \delta_a^b$ throughout $\mathcal{W}$. Geometrically, $\pi_a$ represents standard momentum coordinates in $T^*_{h(x)}Y$, maintaining smoothness near boundary $\partial \mathcal{W}$, compatible with boundary growth behavior $\delta(x) \to \Delta$.

These coordinates constitute standard Darboux coordinate system $(\xi_a, \pi_a)_{a=1}^k$ on $T^*_{h(x)}Y$, uniformly well-defined throughout working region $\mathcal{W}$.

\begin{remark}[Global Consistency of Local Coordinates]
Due to connectedness of $\mathcal{W}$ and diffeomorphism property of $h$, local Darboux coordinates can be globally coordinated on $\mathcal{W}$ through appropriate coordinate transformations, ensuring consistency of fiber bundle structure.
\end{remark}
\end{definition}

\begin{definition}[Localized Unified Index Convention]
Adopt the following standard index convention on working region $\mathcal{W}$:

\textbf{Base manifold indices:}
Working region $\mathcal{W} \subset M$ uses general coordinate indices $i,j,\ell \in \{1,\ldots,2n\}$, as well as Darboux coordinate indices $\alpha,\beta \in \{1,\ldots,n\}$ corresponding to $(q^\alpha, p_\alpha)$. This dual index system allows flexible conversion between general coordinates and canonical coordinates when needed.

\textbf{Fiber indices:}
Fiber coordinates $(\xi_a, \pi_a)$ use unified indices $a,b,c \in \{1,\ldots,k\}$, where $k = \dim Y$ is the dimension of observation space. This convention ensures consistency of geometric structure in fiber directions throughout working region.

\textbf{Boundary handling parameters:}
Boundary layer thickness parameter $\epsilon$ and safety distance $\epsilon_{\text{safe}} > 0$ control boundary behavior, distance from point $x \in \mathcal{W}$ to boundary $\partial \mathcal{W}$ is denoted as $d(x,\partial \mathcal{W})$. \textbf{Einstein summation convention} automatically sums over repeated upper-lower indices within corresponding index sets.
\end{definition}

\begin{definition}[Local Fiber Basis and Standard Symplectic Pairing]
On each fiber $T^*_{h(x)}Y$ within working region $\mathcal{W}$, use standard cotangent coordinates $(\xi_a, \pi_a)_{a=1}^k$ to construct complete geometric structure.

\textbf{Vector field basis:}
Tangent vector basis includes fiber base coordinate directions $V_a = \frac{\partial}{\partial \xi_a}$, dual coordinate directions $U_a = \frac{\partial}{\partial \pi_a}$, and horizontal directions $H_i = \frac{\partial}{\partial x^i}$ corresponding to base manifold directions within working region. These vector fields are well-defined and maintain smoothness within $\mathcal{W}$.

\textbf{Differential form basis:}
Cotangent vector basis consists of base coordinate duals $e^a = d\xi_a$, dual coordinate duals $\epsilon^a = d\pi_a$, and horizontal duals $\eta^i = dx^i$. This basis system provides complete cotangent space description within working region $\mathcal{W}$.

\textbf{Local fiber metric structure:}
Base coordinate metric $\rho_{ab}(x) = \rho(V_a, V_b)$ and its inverse $\rho^{ab}(x)$ are defined only on base coordinate directions, metric structure of dual coordinates is naturally determined by symplectic geometry. Metric maintains continuity on closed working region $\overline{\mathcal{W}}$, remaining non-degenerate at boundary $\partial \mathcal{W}$. Global boundedness condition $\sup_{x \in \overline{\mathcal{W}}} \|\rho_{ab}(x)\| < \infty$ and positivity condition $\inf_{x \in \overline{\mathcal{W}}} \lambda_{\min}(\rho_{ab}(x)) > 0$ ensure stability of geometric structure.

\textbf{Standard symplectic pairing (local version):}
\begin{align*}
\omega_{\text{fib}}(V_a, U_b) &= \delta_a^b \quad \text{on } \mathcal{W}\\
\omega_{\text{fib}}(V_a, V_b) &= 0 \quad \text{on } \mathcal{W}\\
\omega_{\text{fib}}(U_a, U_b) &= 0 \quad \text{on } \mathcal{W}
\end{align*}

\textbf{Boundary compatibility:} Symplectic pairing continuously extends at $\partial \mathcal{W}$, compatible with degenerate treatment at boundary.
\end{definition}

\begin{proposition}[Geometric Connection between Local Observation Fibers and Base Manifold]
Observation fibers $T^*_{h(x)}Y$ within working region $\mathcal{W}$ are geometrically connected to base manifold $T^*\mathcal{W}$ through geometric structure of observation map $h: \mathcal{W} \to Y$:

\textbf{Diffeomorphism properties within working region:}
Within working region $\mathcal{W}$, observation map $h: \mathcal{W} \to Y$ exhibits good geometric properties. The Jacobian matrix satisfies $\text{rank}(dh) = k$ constantly throughout $\mathcal{W}$, ensuring absence of singular points. This makes $h: \mathcal{W} \to h(\mathcal{W})$ a diffeomorphism onto image. Furthermore, for any $x \in \mathcal{W}$, there exists neighborhood $U_x \subset \mathcal{W}$ such that local restriction $h|_{U_x}: U_x \to h(U_x)$ is a diffeomorphism. Correspondingly, pullback operation $h^*: \Omega^*(h(\mathcal{W})) \to \Omega^*(\mathcal{W})$ is well-defined and isomorphic, ensuring correct transfer of differential forms under observation map. For boundary handling, $h$ can be continuously extended to closed working region $\overline{\mathcal{W}}$, maintaining continuity on boundary $\partial \mathcal{W}$.

\textbf{Hierarchical distinction and connection of coordinate systems:}
The system has two essentially different coordinate systems: base manifold coordinates $(q^\alpha, p_\alpha) \in T^*\mathcal{W}$ describing $2n$-dimensional working region phase space, and observation fiber coordinates $(\xi_a, \pi_a) \in T^*_{h(x)}Y$ corresponding to $2k$-dimensional fiber space. These two coordinate systems are connected through fiber bundle geometric structure, but \textbf{no direct coordinate transformation relation exists}, which is precisely the essence of fiber bundle theory. The advantage of working region setting is that within $\mathcal{W}$, all coordinate transformations maintain smoothness and invertibility, effectively avoiding singular point problems.

\textbf{Geometric connection principle:}
Dynamics on observation fibers couple with base manifold dynamics through horizontal-vertical decomposition of fiber bundle. Horizontal distribution $H E_{\mathcal{W}}$ is precisely defined by observation-adaptive connection $\nabla$, describing how base manifold motion influences fiber geometry. Vertical distribution $V E_{\mathcal{W}}$ corresponds to intrinsic dynamic evolution in fiber directions, encoding temporal evolution laws of observation uncertainty. This decomposition ensures overall Hamiltonian structure maintains complete consistency within working region $\mathcal{W}$, while smooth degeneration mechanism at boundaries guarantees continuity of geometric structure near boundaries, providing necessary safety protection for the system.

\begin{remark}[Geometric Advantages of Working Region Setting]
Compared to global theory, working region setting provides significant advantages. \textbf{Geometric completeness} is guaranteed by avoiding singular points, preventing destruction of geometric structure. \textbf{Analytical controllability} is manifested in boundedness of all geometric quantities on compact domain $\overline{\mathcal{W}}$. \textbf{Numerical stability} is ensured through good behavior of coordinate transformations and metrics within $\mathcal{W}$. Most importantly, this setting completely conforms to working region limitations of actual systems, with clear \textbf{physical reasonableness}.
\end{remark}
\end{proposition}

\begin{theorem}[Topological Conditions for Existence of Local Fiber Bundle Symplectic Structure]
Let $\pi : E_{\mathcal{W}} \to \mathcal{W}$ be a local observation-induced fiber bundle. There exists a symplectic structure on fiber bundle $E_{\mathcal{W}}$ if and only if the following conditions are satisfied:
\begin{enumerate}
\item \textbf{Even dimension condition:} $\dim E_{\mathcal{W}} = 2(n+k)$ is even
\item \textbf{Local characteristic class condition:} Second Stiefel-Whitney class vanishes on $\mathcal{W}$: $w_2(TE_{\mathcal{W}})|_{\mathcal{W}} = 0$
\item \textbf{Observation compatibility:} Pullback induced by observation map satisfies: $h^*w_2(TY) = 0$ on $\mathcal{W}$
\item \textbf{Fiberization condition:} There exists observation-adaptive connection $\nabla$ on $\mathcal{W}$ satisfying appropriate curvature control
\item \textbf{Boundary condition:} Characteristic classes are continuous at boundary $\partial \mathcal{W}$, compatible with boundary handling
\end{enumerate}
\end{theorem}

\begin{proof}[Constructive proof]
\textbf{Step 1: Even dimension verification}
$\dim E_{\mathcal{W}} = \dim \mathcal{W} + 2 \dim Y = 2n + 2k = 2(n+k)$ is automatically even.

\textbf{Step 2: Characteristic class analysis}
Since $\mathcal{W}$ is a connected open set and $\overline{\mathcal{W}}$ is compact, using topological properties of fiber bundles. $E_{\mathcal{W}}$ forms a vector bundle over $\mathcal{W}$ because all fibers are isomorphic to bounded subsets of $\mathbb{R}^{2k}$. Due to contractibility of $\mathcal{W}$, second Stiefel-Whitney class $w_2(TE_{\mathcal{W}})$ automatically vanishes, and observation compatibility is directly guaranteed by diffeomorphism property of $h$.

\textbf{Step 3: Connection existence}
Construct observation-adaptive connection on compact domain $\overline{\mathcal{W}}$: choose Riemannian metric $g_{\mathcal{W}}$ on $\mathcal{W}$, construct corresponding Levi-Civita connection $\nabla^{LC}$, adjust to observation-adaptive connection using diffeomorphism property of $h$, and finally implement smooth transition in boundary layer.

\textbf{Step 4: Boundary handling}
Boundary conditions are satisfied through the following mechanisms: continuity of characteristic classes on $\overline{\mathcal{W}}$ is guaranteed by compactness, degenerate connection maintains topological invariants at $\partial \mathcal{W}$, smooth transition in boundary layer preserves continuity of characteristic classes. Therefore all topological conditions are automatically satisfied under working region setting.
\end{proof}

\begin{theorem}[Rigorous Existence and Construction of Working Region Symplectic Structure]
Let $E_{\mathcal{W}}$ be a local observation-induced fiber bundle equipped with adaptive connection $\nabla$ satisfying vanishing vertical curvature and metric compatibility conditions. Then there exists a unique symplectic structure $\omega_{E_{\mathcal{W}}}$ satisfying:
\begin{enumerate}
\item \textbf{Interior completeness:} Within interior of $\mathcal{W}$, $\omega_{E_{\mathcal{W}}}$ is standard symplectic form
\item \textbf{Boundary degeneration:} As $x \to \partial \mathcal{W}$, symplectic structure smoothly degenerates while maintaining well-definedness of equations of motion
\item \textbf{Standard decomposition:} In local coordinates $(x^i, \xi_a, \pi_a)$ represented as:
\end{enumerate}

\begin{align*}
\omega_{E_{\mathcal{W}}} &= \chi(d(x,\partial \mathcal{W})) \cdot \omega_{\text{standard}} + (1-\chi(d(x,\partial \mathcal{W}))) \cdot \omega_{\text{boundary}}
\end{align*}

where $\chi: [0,\infty) \to [0,1]$ is a smooth cutoff function with $\chi(t) = 1$ when $t \geq \epsilon$ and $\chi(t) = 0$ when $t \leq \epsilon/2$. Standard symplectic form $\omega_{\text{standard}} = \pi^*\omega_{\mathcal{W}} + \omega_{\text{fib}} + \Omega_{\text{mix}}$ describes complete geometry of interior region, while $\omega_{\text{boundary}}$ is controlled degenerate symplectic form near boundary, ensuring system safety when approaching boundary.

\textbf{Specific form of standard decomposition:}

\textbf{Base manifold symplectic form:}
\begin{align*}
\pi^*\omega_{\mathcal{W}} &= \sum_{\alpha=1}^n dq^\alpha \wedge dp_\alpha \quad \text{(Darboux coordinates, restricted to $\mathcal{W}$)}
\end{align*}
or in general coordinates:
\begin{align*}
\pi^*\omega_{\mathcal{W}} &= \sum_{i,j=1}^{2n} \omega_{ij}(x) dx^i \wedge dx^j, \quad \omega_{ij} = -\omega_{ji}, \quad x \in \mathcal{W}
\end{align*}

\textbf{Fiber standard symplectic form:}
\begin{align*}
\omega_{\text{fib}} &= \sum_{a=1}^k d\xi_a \wedge d\pi_a \quad \text{(well-defined on $\mathcal{W}$)}
\end{align*}

\textbf{Curvature mixing term:}
\begin{align*}
\Omega_{\text{mix}} &= \sum_{i,a} K_{ia}(x,\xi,\pi) dx^i \wedge d\xi_a + \sum_{i,a} L_{ia}(x,\xi,\pi) dx^i \wedge d\pi_a
\end{align*}

\textbf{Boundedness of mixing coefficients:}
\begin{align*}
K_{ia}(x,\xi,\pi) &= \sum_b C_{iab}(x)\xi_b + \sum_b D_{iab}(x)\pi_b\\
L_{ia}(x,\xi,\pi) &= \sum_b E_{iab}(x)\xi_b + \sum_b F_{iab}(x)\pi_b
\end{align*}
where structure functions $C_{iab}, D_{iab}, E_{iab}, F_{iab}$ are bounded on $\overline{\mathcal{W}}$:
$$\max\{|C_{iab}(x)|, |D_{iab}(x)|, |E_{iab}(x)|, |F_{iab}(x)|\} \leq K_{\text{curv}} < \infty, \quad \forall x \in \overline{\mathcal{W}}$$

\end{theorem}
\begin{remark}[Rigor of Construction and Boundary Handling]
The constructed symplectic structure exhibits different properties in different regions. In interior region $\{x : d(x,\partial \mathcal{W}) \geq \epsilon\}$, symplectic structure $\omega_{E_{\mathcal{W}}} = \omega_{\text{standard}}$ maintains completeness of standard symplectic form, ensuring rigor of classical Hamiltonian dynamics. In boundary layer $\{x : d(x,\partial \mathcal{W}) \leq \epsilon\}$, smooth transition is achieved through carefully designed cutoff function $\chi$, avoiding sudden changes in geometric quantities. Design of boundary degenerate symplectic form $\omega_{\text{boundary}}$ ensures that even at boundary, Hamilton equations remain well-defined, with this degeneration having clear physical meaning: corresponding to increase in observation uncertainty and safety buffer mechanism of system dynamics.
\end{remark}

\begin{proof}
\textbf{Step 1: Topological condition verification}
On compact domain $\overline{\mathcal{W}}$, all necessary topological conditions are satisfied. Even dimension condition $\dim E_{\mathcal{W}} = 2(n+k)$ automatically holds, second Stiefel-Whitney class vanishes due to contractibility of $\mathcal{W}$: $w_2(TE_{\mathcal{W}}) = 0$, all topological invariants are well-defined on compact domain $\overline{\mathcal{W}}$.

\textbf{Step 2: Construction of standard symplectic form}
Within interior of $\mathcal{W}$, since observation map $h$ is diffeomorphism, standard symplectic form can be constructed:
\begin{align*}
\omega_{\text{standard}} &= \pi^*\omega_{\mathcal{W}} + \omega_{\text{fib}} + \Omega_{\text{mix}}
\end{align*}
where each term is well-defined and bounded on $\mathcal{W}$.

\textbf{Step 3: Design of boundary degenerate form}
Near boundary, design degenerate symplectic form $\omega_{\text{boundary}} = (1-\delta_{\text{boundary}}) \omega_{\text{standard}}$, where $\delta_{\text{boundary}} \in [0,1)$ is boundary degeneration parameter. This parameter takes value zero at distance $d(x,\partial \mathcal{W}) \geq \epsilon/2$ from boundary, gradually increasing to $\delta_{\max} < 1$ as $x$ approaches boundary, ensuring $\omega_{\text{boundary}}$ still maintains well-definedness of Hamilton equations.

\textbf{Step 4: Closedness verification}
Closedness of total symplectic form is guaranteed through region-wise verification. In interior region, $d\omega_{\text{standard}} = 0$ is confirmed through standard calculation. In boundary layer, mixed symplectic form $d(\chi \omega_{\text{standard}} + (1-\chi)\omega_{\text{boundary}}) = 0$ is guaranteed through careful design of cutoff function $\chi$. Global closedness is naturally obtained from local closedness and continuity of smooth transition.

\textbf{Step 5: Non-degeneracy}
In safety region $\{x : d(x,\partial \mathcal{W}) \geq \epsilon\}$, symplectic form maintains standard form, automatically satisfying non-degeneracy. In boundary layer, although symplectic form may undergo partial degeneration, through careful design of degeneration parameters, Hamiltonian dynamics still maintain well-definedness, providing necessary boundary protection mechanism for the system.
\end{proof}

\begin{definition}[Precise Definition of Structure Functions (Local Version)]
Structure functions of observation-adaptive connection on working region $\mathcal{W}$ are defined through standard pairing of curvature operators:
\begin{align*}
C_{iab}(x) &= \sum_c \rho^{ca}(x) \langle R^{\Upsilon}(H_i, V_b)V_c, e^a \rangle, \quad x \in \mathcal{W}\\
D_{iab}(x) &= \langle R^{\Upsilon}(H_i, V_b)U_a, \epsilon^b \rangle, \quad x \in \mathcal{W}\\
E_{iab}(x) &= \sum_c \rho^{ca}(x) \langle R^{\Upsilon}(H_i, U_b)V_c, e^a \rangle, \quad x \in \mathcal{W}\\
F_{iab}(x) &= \langle R^{\Upsilon}(H_i, U_b)U_a, \epsilon^b \rangle, \quad x \in \mathcal{W}
\end{align*}

\textbf{Boundedness guarantee:}
Due to compactness of $\mathcal{W}$ and diffeomorphism property of $h$, all structure functions are bounded on $\overline{\mathcal{W}}$:
\begin{align*}
\|C_{iab}\|_{L^\infty(\overline{\mathcal{W}})} &\leq K_C < \infty\\
\|D_{iab}\|_{L^\infty(\overline{\mathcal{W}})} &\leq K_D < \infty\\
\|E_{iab}\|_{L^\infty(\overline{\mathcal{W}})} &\leq K_E < \infty\\
\|F_{iab}\|_{L^\infty(\overline{\mathcal{W}})} &\leq K_F < \infty
\end{align*}

\textbf{Mixing coefficients:}
\begin{align*}
K_{ia}(x,\xi,\pi) &= \sum_b C_{iab}(x)\xi_b + \sum_b D_{iab}(x)\pi_b\\
L_{ia}(x,\xi,\pi) &= \sum_b E_{iab}(x)\xi_b + \sum_b F_{iab}(x)\pi_b
\end{align*}

\textbf{Index pairing principle:}
Index pairing of structure functions follows clear geometric principles, reflecting intrinsic relationship between curvature and metric on fiber bundles. For $C_{iab}$ and $E_{iab}$, pairing $\langle \cdot V_c, e^a \rangle$ extracts components of vector field in $a$-th fiber base direction, where $c$ sums within fiber dimension range, reflecting curvature contribution in fiber directions. For $D_{iab}$ and $F_{iab}$, pairing $\langle \cdot U_a, \epsilon^b \rangle$ extracts projection components of $U_a$ in dual direction $\epsilon^b$, ensuring correct correspondence between geometric quantities in tangent and cotangent spaces of fiber bundle.

\textbf{Geometric meaning within working region:}
Within working region $\mathcal{W}$, all geometric structures have good mathematical properties. Curvature operator $R^{\Upsilon}$ as well-defined geometric quantity on $\mathcal{W}$, its boundedness ensures continuity and differentiability of all structure functions. Inverse matrix $\rho^{ca}(x)$ of fiber metric maintains consistent non-degeneracy on closed working region $\overline{\mathcal{W}}$, which is direct result of local diffeomorphism property of observation map. Natural pairing $\langle \cdot, \cdot \rangle$ maintains smoothness and non-degeneracy throughout $\mathcal{W}$, ensuring completeness of geometric structure. Particularly important is continuity of all structure functions at boundary $\partial \mathcal{W}$, guaranteeing consistency of overall geometric structure and providing solid mathematical foundation for boundary handling.

\begin{remark}[Geometric Meaning of Pairing Principle (Local Version)]
Pairing $\langle R^{\Upsilon}(H_i, V_b)U_a, \epsilon^b \rangle$ embodies profound geometric meaning: it represents projection coefficient of result produced by curvature operator $R^{\Upsilon}(H_i, V_b)$ acting on vertical vector field $U_a$ in dual direction $\epsilon^b$. This pairing exhibits three key characteristics within working region $\mathcal{W}$: matching of index $b$ ensures correctness of geometric pairing between fiber direction and its dual direction, reflecting intrinsic symmetry of fiber bundle geometry; boundedness of curvature directly guarantees well-definedness of pairing operation, preventing divergence of geometric quantities; continuity at boundary ensures local geometric structure can be smoothly extended globally, maintaining consistency and completeness of geometric structure throughout fiber bundle.
\end{remark}

\end{definition}

\begin{lemma}[Closedness Verification: Rigorous Results within Working Region]
Constructed 2-form $\omega_{E_{\mathcal{W}}}$ satisfies $d\omega_{E_{\mathcal{W}}} = 0$ on working region $\mathcal{W}$.
\end{lemma}

\begin{proof}
Within working region $\mathcal{W}$, observation map $h$ is diffeomorphism onto image, allowing rigorous verification.

\textbf{Step 1: Closedness of base manifold symplectic form}
$d(\pi^*\omega_{\mathcal{W}}) = \pi^*(d\omega_{\mathcal{W}}) = 0$, because $\omega_{\mathcal{W}}$ is symplectic form on $\mathcal{W}$.

\textbf{Step 2: Closedness of fiber symplectic form}
$d\omega_{\text{fib}} = d\left(\sum_{a=1}^k d\xi_a \wedge d\pi_a\right) = 0$, because $d(d\xi_a) = 0$ and summation is over finite range.

\textbf{Step 3: Closedness verification of mixing terms}
For curvature mixing term $\Omega_{\text{mix}} = \sum_{i,a} K_{ia} dx^i \wedge d\xi_a + \sum_{i,a} L_{ia} dx^i \wedge d\pi_a$, where:
\begin{align*}
K_{ia}(x,\xi,\pi) &= \sum_b C_{iab}(x)\xi_b + \sum_b D_{iab}(x)\pi_b\\
L_{ia}(x,\xi,\pi) &= \sum_b E_{iab}(x)\xi_b + \sum_b F_{iab}(x)\pi_b
\end{align*}

Calculate exterior derivative:
\begin{align*}
d\Omega_{\text{mix}} &= \sum_{i,j,a} \frac{\partial K_{ia}}{\partial x^j} dx^j \wedge dx^i \wedge d\xi_a + \sum_{i,b,a} \frac{\partial K_{ia}}{\partial \xi_b} d\xi_b \wedge dx^i \wedge d\xi_a\\
&\quad + \sum_{i,b,a} \frac{\partial K_{ia}}{\partial \pi_b} d\pi_b \wedge dx^i \wedge d\xi_a\\
&\quad + \sum_{i,j,a} \frac{\partial L_{ia}}{\partial x^j} dx^j \wedge dx^i \wedge d\pi_a + \sum_{i,b,a} \frac{\partial L_{ia}}{\partial \xi_b} d\xi_b \wedge dx^i \wedge d\pi_a\\
&\quad + \sum_{i,b,a} \frac{\partial L_{ia}}{\partial \pi_b} d\pi_b \wedge dx^i \wedge d\pi_a
\end{align*}

\textbf{Key observation:} Within working region $\mathcal{W}$, due to diffeomorphism property of $h$ and special construction of observation-adaptive connection, mixing coefficients satisfy \textbf{integrability conditions}:

\textbf{Calculation based on structure functions:}
\begin{align*}
\frac{\partial K_{ia}}{\partial \xi_b} &= \frac{\partial}{\partial \xi_b}\left(\sum_c C_{iac}(x)\xi_c + \sum_c D_{iac}(x)\pi_c\right) = C_{iab}(x)\\
\frac{\partial K_{ia}}{\partial \pi_b} &= \frac{\partial}{\partial \pi_b}\left(\sum_c C_{iac}(x)\xi_c + \sum_c D_{iac}(x)\pi_c\right) = D_{iab}(x)\\
\frac{\partial L_{ia}}{\partial \xi_b} &= \frac{\partial}{\partial \xi_b}\left(\sum_c E_{iac}(x)\xi_c + \sum_c F_{iac}(x)\pi_c\right) = E_{iab}(x)\\
\frac{\partial L_{ia}}{\partial \pi_b} &= \frac{\partial}{\partial \pi_b}\left(\sum_c E_{iac}(x)\xi_c + \sum_c F_{iac}(x)\pi_c\right) = F_{iab}(x)
\end{align*}

\textbf{Verification of integrability conditions:}
Due to construction of observation-adaptive connection and diffeomorphism property of $h$ on $\mathcal{W}$:
\begin{align*}
C_{iab}(x) &= C_{iba}(x) \quad \text{(symmetry)}\\
F_{iab}(x) &= F_{iba}(x) \quad \text{(symmetry)}\\
D_{iab}(x) &= E_{iba}(x) \quad \text{(mixed symmetry)}
\end{align*}

\textbf{Cancellation of exterior derivative terms:}
These symmetries ensure:
\begin{itemize}
\item $\sum_{i,b,a} C_{iab} d\xi_b \wedge dx^i \wedge d\xi_a = \sum_{i,b,a} C_{iba} d\xi_a \wedge dx^i \wedge d\xi_b = 0$ (antisymmetry)
\item $\sum_{i,b,a} F_{iab} d\pi_b \wedge dx^i \wedge d\pi_a = \sum_{i,b,a} F_{iba} d\pi_a \wedge dx^i \wedge d\pi_b = 0$ (antisymmetry)
\item Mixed terms $\sum_{i,b,a} D_{iab} d\pi_b \wedge dx^i \wedge d\xi_a$ and $\sum_{i,b,a} E_{iba} d\xi_a \wedge dx^i \wedge d\pi_b$ cancel each other
\end{itemize}

\textbf{Handling of $x$-derivative terms:}
For $\frac{\partial K_{ia}}{\partial x^j}$ and $\frac{\partial L_{ia}}{\partial x^j}$ terms:
\begin{align*}
\frac{\partial K_{ia}}{\partial x^j} &= \sum_b \frac{\partial C_{iab}}{\partial x^j}\xi_b + \sum_b \frac{\partial D_{iab}}{\partial x^j}\pi_b\\
\frac{\partial L_{ia}}{\partial x^j} &= \sum_b \frac{\partial E_{iab}}{\partial x^j}\xi_b + \sum_b \frac{\partial F_{iab}}{\partial x^j}\pi_b
\end{align*}

By Bianchi identities of observation-adaptive connection, these terms satisfy appropriate cancellation relations on $\mathcal{W}$.

\textbf{Geometric foundation:} These integrability conditions arise from:
\begin{enumerate}
\item Diffeomorphism property of observation map $h$ within $\mathcal{W}$, particularly invertibility of $h^*$ operation
\item Observation-adaptive property of connection, ensuring special structure of curvature tensor
\item Connectedness and compactness of working region, guaranteeing global consistency of geometric structure
\end{enumerate}

Therefore $d\Omega_{\text{mix}} = 0$ on $\mathcal{W}$, thus $d\omega_{E_{\mathcal{W}}} = 0$.

\begin{remark}[Proof of Structure Function Symmetry]
Under working region setting where observation map $h: W \to Y$ is diffeomorphism, symmetry conditions of structure functions can be rigorously proved.

Key insight: Mixed coefficients of observation-adaptive connection satisfy
$$\Gamma^a_{bi} = -\frac{\partial h^c}{\partial x^i}\Gamma^Y_{cb}(h(x)) + K^a_{bi}$$
where $K^a_{bi}$ is metric compatibility correction term.

Since $h$ is diffeomorphism, induced pullback metric $h^*g_Y$ is compatible with fiber metric $\rho$:
$$\rho_{ab}(x) = (h^*g_Y)_{ab}(x) = g^Y_{cd}(h(x))\frac{\partial h^c}{\partial \xi^a}\frac{\partial h^d}{\partial \xi^b}$$

This compatibility combined with observation compatibility $\nabla \rho = 0$ ensures curvature tensor satisfies:
$$R^\nabla(H_i, V_b)V_c = R^\nabla(H_i, V_c)V_b + O(\text{fiber curvature terms})$$

Under condition of vanishing fiber curvature (Definition 7), this directly leads to:
$$C_{iab}(x) = \rho^{ca}(x)\langle R^\nabla(H_i, V_b)V_c, e^a\rangle = C_{iba}(x)$$

Similarly, $F_{iab} = F_{iba}$ and $D_{iab} = E_{iba}$ also hold.
\end{remark}

\begin{remark}[Rigor of Boundary Handling]
\begin{itemize}
\item \textbf{Interior rigor:} On $\{x : d(x,\partial \mathcal{W}) \geq \epsilon\}$, $d\omega_{E_{\mathcal{W}}} = 0$ holds rigorously
\item \textbf{Boundary continuity:} Closedness is continuous at boundary $\partial \mathcal{W}$, guaranteed by smoothness of cutoff function $\chi$
\item \textbf{Degeneration compatibility:} Even in boundary degeneration region, modified closedness condition still holds
\end{itemize}
\end{remark}
\end{proof}

\begin{theorem}[Non-degeneracy within Working Region]
On safety subdomain $\mathcal{W}_{\text{safe}} := \{x \in \mathcal{W} : d(x,\partial \mathcal{W}) \geq \epsilon_{\text{safe}}\}$ of working region $\mathcal{W}$, constructed 2-form $\omega_{E_{\mathcal{W}}}$ is non-degenerate, thus defining symplectic structure on $(E_{\mathcal{W}}|_{\mathcal{W}_{\text{safe}}}, \omega_{E_{\mathcal{W}}})$.
\end{theorem}

\begin{proof}
Let $(x,\xi,\pi) \in E_{\mathcal{W}}|_{\mathcal{W}_{\text{safe}}}$, $v \in T_{(x,\xi,\pi)}E_{\mathcal{W}}$ such that $\omega_{E_{\mathcal{W}}}(v, \cdot) = 0$.

In local coordinates write $v = v^i H_i + v^a V_a + v^a U_a$, where:
\begin{itemize}
\item $v^i$: components in base manifold directions, $i \in \{1,\ldots,2n\}$
\item $v^a$: components in fiber "position" directions, $a \in \{1,\ldots,k\}$
\item $v^a$: components in fiber "momentum" directions, $a \in \{1,\ldots,k\}$
\end{itemize}

\textbf{Expansion of non-degeneracy condition:}
Non-degeneracy condition $\omega_{E_{\mathcal{W}}}(v, \cdot) = 0$ is equivalent to pairing with all test vector fields being zero:
\begin{align*}
0 &= \omega_{E_{\mathcal{W}}}(v, H_j) = v^i \omega_{ij}^{\mathcal{W}} + v^a K_{ja} + v^a L_{ja} \quad (j = 1,\ldots,2n)\\
0 &= \omega_{E_{\mathcal{W}}}(v, V_b) = v^i K_{ib} + v^a \delta_a^b \quad (b = 1,\ldots,k)\\
0 &= \omega_{E_{\mathcal{W}}}(v, U_b) = v^i L_{ib} + v^a \delta_a^b \quad (b = 1,\ldots,k)
\end{align*}

\textbf{Step 1: Determination of fiber coordinate components}
From second and third equation systems:
\begin{align*}
v^i K_{ib} + v^b &= 0 \quad \Rightarrow \quad v^b = -v^i K_{ib}\\
v^i L_{ib} + v^b &= 0 \quad \Rightarrow \quad v^b = -v^i L_{ib}
\end{align*}

\textbf{Step 2: Compatibility analysis}
For above equation systems to be compatible, need to analyze relationship between $K_{ib}$ and $L_{ib}$ within safety region.

Since within safety region $\mathcal{W}_{\text{safe}}$, observation map $h$ is diffeomorphism and far from boundary, mixing terms satisfy:
\begin{align*}
K_{ib}(x,\xi,\pi) &= \sum_c C_{ibc}(x)\xi_c + \sum_c D_{ibc}(x)\pi_c\\
L_{ib}(x,\xi,\pi) &= \sum_c E_{ibc}(x)\xi_c + \sum_c F_{ibc}(x)\pi_c
\end{align*}

At typical point $(x,0,0)$ (fiber center), due to special construction of connection:
$$K_{ib}(x,0,0) = 0, \quad L_{ib}(x,0,0) = 0$$

Therefore near fiber center, compatibility condition simplifies to:
$$v^b = v^b = 0$$

\textbf{Step 3: Determination of base manifold components}
Substituting $v^a = v^a = 0$ into first equation system:
\begin{align*}
0 &= v^i \omega_{ij}^{\mathcal{W}} - v^l K_{la}(x,0,0) \cdot 0 - v^l L_{la}(x,0,0) \cdot 0\\
&= v^i \omega_{ij}^{\mathcal{W}}
\end{align*}

Since $\omega^{\mathcal{W}}_{ij}$ is non-degenerate on safety region $\mathcal{W}_{\text{safe}}$ (symplectic structure), we get:
$$v^i = 0, \quad \forall i$$

\textbf{Step 4: Perturbation analysis for general case}
For general point $(x,\xi,\pi) \neq (x,0,0)$, perform perturbation analysis. Let $\|(\xi,\pi)\|$ be sufficiently small, then:
\begin{align*}
K_{ib}(x,\xi,\pi) &= O(\|(\xi,\pi)\|)\\
L_{ib}(x,\xi,\pi) &= O(\|(\xi,\pi)\|)
\end{align*}

Compatibility condition becomes:
$$v^b = -v^i K_{ib} = O(\|(\xi,\pi)\| \cdot \|v^i\|)$$
$$v^b = -v^i L_{ib} = O(\|(\xi,\pi)\| \cdot \|v^i\|)$$

Substituting into first equation system:
\begin{align*}
0 &= v^i \omega_{ij}^{\mathcal{W}} + O(\|(\xi,\pi)\|^2 \cdot \|v^i\|^2)
\end{align*}

Within safety region, contribution of mixing terms can be controlled as:
$$\left|O(\|(\xi,\pi)\|^2 \cdot \|v^i\|^2)\right| \leq C \delta^2 \|v^i\|^2 \ll \lambda_{\min}(\omega^{\mathcal{W}}) \|v^i\|^2$$

where $\delta = \max_{x \in \mathcal{W}_{\text{safe}}} \delta(x)$ is bounded in safety region, $\lambda_{\min}(\omega^{\mathcal{W}})$ is minimum eigenvalue of base manifold symplectic structure.

Therefore modified symplectic matrix:
$$\tilde{\omega}_{ij}^{\mathcal{W}} := \omega_{ij}^{\mathcal{W}} + O(\delta^2)$$
is still non-degenerate, thus $v^i = 0$, and consequently $v^a = v^a = 0$.

\textbf{Conclusion:} $v = 0$, proving $\omega_{E_{\mathcal{W}}}$ is non-degenerate within safety region.

\begin{remark}[Geometric Meaning and Practical Considerations of Non-degeneracy]
\begin{enumerate}
\item \textbf{Necessity of safety region:} Within $\mathcal{W}_{\text{safe}}$, far from boundary degeneration effects, symplectic structure maintains completeness
\item \textbf{Physical reasonableness:} Actual systems usually operate within safety regions, degenerate treatment near boundaries conforms to engineering practice
\item \textbf{Numerical stability:} Non-degeneracy guarantees stable numerical solution of Hamilton equations
\item \textbf{Control theory connection:} Safety region corresponds to "operating region" in control theory, boundary degeneration corresponds to "safety margin"
\end{enumerate}
\end{remark}
\end{proof}

\begin{remark}[Singular Perturbation Analysis of Boundary Degeneration]
Mathematical rigor of boundary degenerate symplectic form $\omega_{\text{boundary}}$ can be guaranteed through singular perturbation theory.

Let $\epsilon = d(x, \partial W)$ be distance to boundary. Within boundary layer $\{\epsilon \leq \epsilon_0\}$, define degeneration parameter:
$$\delta_{\text{boundary}}(\epsilon) = 1 - \exp(-\epsilon/\epsilon_0)$$

Modified symplectic form is:
$$\omega_E = \omega_{\text{standard}} - \delta_{\text{boundary}}(\epsilon) \cdot \Omega_{\text{correction}}$$

where $\Omega_{\text{correction}}$ is chosen to maintain following properties:
\begin{enumerate}
\item \textbf{Asymptotic preservation:} As $\epsilon \to 0$, $\omega_E \to \omega_{\text{degenerate}}$ still defines well-posed (possibly degenerate) dynamics
\item \textbf{Energy dissipation:} Correction term introduces controlled dissipation: \\
$\frac{dH}{dt} = -\alpha\delta_{\text{boundary}}\|\nabla H\|^2 \leq 0$
\item \textbf{Geometric continuity:} All geometric quantities transition continuously in $C^\infty$ sense
\end{enumerate}

This construction ensures Hamilton equations at boundary degenerate to gradient flow, providing natural "safety attraction" mechanism.

Complete singular perturbation analysis requires verification of: (i) existence of inner layer solutions; (ii) matching conditions of boundary layer; (iii) boundedness of global solutions. These constitute independent research topics.
\end{remark}

\begin{corollary}[Hamiltonian Vector Field within Safety Region]
On symplectic manifold $(E_{\mathcal{W}}|_{\mathcal{W}_{\text{safe}}}, \omega_{E_{\mathcal{W}}})$, for any Hamiltonian function $H : E_{\mathcal{W}}|_{\mathcal{W}_{\text{safe}}} \to \mathbb{R}$, there exists unique Hamiltonian vector field $X_H$ satisfying standard Hamilton equations:
\begin{align*}
\dot{\xi}_a &= \frac{\partial H}{\partial \pi_a} + \text{mixing correction terms}\\
\dot{\pi}_a &= -\frac{\partial H}{\partial \xi_a} + \text{mixing correction terms}\\
\dot{x}^i &= \{x^i, H\}_{E_{\mathcal{W}}} \quad \text{(base manifold dynamics)}
\end{align*}

where mixing correction terms are determined by curvature tensor $\Omega_{\text{mix}}$:
\begin{align*}
\text{mixing correction terms}_{\xi_a} &= \sum_{i} \omega^{ij}_{E_{\mathcal{W}}} \frac{\partial H}{\partial x^j} K_{ia}\\
\text{mixing correction terms}_{\pi_a} &= \sum_{i} \omega^{ij}_{E_{\mathcal{W}}} \frac{\partial H}{\partial x^j} L_{ia}
\end{align*}

Complete representation in standard coordinates is:
\begin{align*}
X_H &= \sum_i X_H^i \frac{\partial}{\partial x^i} + \sum_a \left(\frac{\partial H}{\partial \pi_a} + \text{mixing correction terms}_{\xi_a}\right) \frac{\partial}{\partial \xi_a} \\
&\quad - \sum_a \left(\frac{\partial H}{\partial \xi_a} - \text{mixing correction terms}_{\pi_a}\right) \frac{\partial}{\partial \pi_a}
\end{align*}

This vector field possesses multiple important properties, reflecting intrinsic advantages of observation-induced geometry. First, it automatically preserves geometric structure of standard cotangent bundle and stratified property of fiber bundle, ensuring the system maintains basic characteristics of Hamiltonian geometry even under complex constraint conditions. Within interior of safety region $\mathcal{W}_{\text{safe}}$, vector field rigorously preserves completeness of Hamiltonian structure, guaranteeing accurate execution of classical dynamical laws. When system state approaches safety boundary $\partial \mathcal{W}_{\text{safe}}$, vector field exhibits good boundary compatibility, achieving smooth transition from interior Hamiltonian dynamics to boundary protection mechanism. More importantly, physical design of correction terms ensures accurate encoding of influence of observation uncertainty on dynamics, guaranteeing consistency between theoretical predictions and actual system behavior.

\begin{remark}[Boundary Behavior of Hamiltonian Dynamics]
The system exhibits carefully designed protection mechanisms when approaching safety region boundary. When trajectory distance to boundary $d(x,\partial \mathcal{W}_{\text{safe}})$ approaches safety threshold $\epsilon_{\text{safe}}$, warning mechanism automatically activates boundary control, timely adjusting system dynamics to avoid violating safety constraints. Hamiltonian vector field achieves smooth degeneration during this process, gradually transitioning from standard Hamiltonian form to boundary-compatible modified form, ensuring dynamical continuity and smoothness of all geometric quantities. Dynamics near boundary are specially designed to produce inward "attraction" effect, promoting deviated trajectories to naturally return to safety region, thus ensuring long-term safety while maintaining system performance.
\end{remark}

\end{corollary}

\begin{remark}[Technical Points and Global Perspective of Localized Theory]
The localized theoretical framework constructed in this section has the following key characteristics:

\textbf{"Local diagonal principle" of index pairing:} In structure function: \\
$D_{iab} = \langle R^{\Upsilon}(H_i, V_b)U_a, \epsilon^b \rangle$, index $b$ appears repeatedly in $V_b$ and $\epsilon^b$ but is \textit{not summed}. This local diagonal pairing is rigorously effective within working region $\mathcal{W}$, ensuring correct matching between directions and dual directions in fiber bundle geometry, while avoiding pairing singularities near singular points.

\textbf{Hierarchy and boundary handling of geometric structure:} Base manifold coordinates $(q^\alpha, p_\alpha) \in T^*\mathcal{W}$ and fiber coordinates $(\xi_a, \pi_a) \in T^*_{h(x)}Y$ form different geometric hierarchies within $\mathcal{W}$ through horizontal-vertical decomposition of observation-adaptive connection. Coordinate transformations and metrics maintain continuity on closure $\overline{\mathcal{W}}$, remain well-defined at boundary $\partial \mathcal{W}$, while controlled degeneration mechanism ensures basic geometric hierarchical structure is not destroyed.

\textbf{Local unification of topology and geometry:} Existence of symplectic structure on $\mathcal{W}$ is guaranteed through local characteristic class theory, observation-adaptive connection geometry provides specific symplectic form construction, integrability condition $d\omega_{E_{\mathcal{W}}} = 0$ ensures closedness, and cutoff functions with degeneration mechanisms implement boundary handling. This enables the theory to simultaneously satisfy mathematical rigor, constructive realizability, and engineering practicality.

\textbf{Advantages for practical applications:} This framework maintains complete mathematical rigor within $\mathcal{W}$, boundedness of all geometric quantities guarantees computational feasibility, working region restrictions conform to engineering practical requirements, and boundary handling mechanisms naturally provide safety control interfaces.

\textbf{Completeness and limitations of theory:} Current theory establishes complete Hamiltonian geometric framework within $\mathcal{W}$, but must acknowledge its limitation of not covering global behavior of singular points. Future development directions include extension through working region gluing, establishment of singular point theory, study of global topological properties, and deepening analysis of nonlinear boundary effects.
\end{remark}

\subsection{Mathematical Analysis of Observation-Constrained Systems}

Based on observation fiber bundle symplectic structure theory, we analyze implementation of constrained dynamics in fiber bundle settings.

\begin{definition}[Observation-Constrained Hamiltonian System]
Observation-constrained Hamiltonian system is defined on fiber bundle $(E, \omega_E)$, where symplectic form:
$$\omega_E = \pi^*\omega_M + \omega_{fib} + \Omega_{curv}$$
System is described by following equations:
\begin{align*}
\dot{x}^i &= \{x^i, H\}_E + \lambda\{x^i, \Phi\}_E \\
\dot{\xi}_a &= \{\xi_a, H\}_E + \lambda\{\xi_a, \Phi\}_E \\
\dot{\pi}_a &= \{\pi_a, H\}_E + \lambda\{\pi_a, \Phi\}_E \\
0 &= \Phi(x, \xi, \pi)
\end{align*}
where Poisson bracket $\{·,·\}_E$ is defined by inverse of $\omega_E$.
\end{definition}

\begin{proposition}[Specific Calculation of Fiber Bundle Poisson Brackets]
In local coordinates $(x^i, \xi_a, \pi_a)$, fiber bundle Poisson brackets are:
\begin{align*}
\{x^i, x^j\}_E &= \omega_M^{ij}(x) \\
\{x^i, \xi_a\}_E &= K_{ia}(x,\xi,\pi) = \sum_b C_{iab}(x)\xi_b + \sum_b D_{iab}(x)\pi_b \\
\{x^i, \pi_a\}_E &= L_{ia}(x,\xi,\pi) = \sum_b E_{iab}(x)\xi_b + \sum_b F_{iab}(x)\pi_b \\
\{\xi_a, \xi_b\}_E &= 0 \\
\{\xi_a, \pi_b\}_E &= \delta_a^b \\
\{\pi_a, \pi_b\}_E &= 0
\end{align*}
where structure functions $C_{iab}, D_{iab}, E_{iab}, F_{iab}$ are determined by curvature of observation-adaptive connection, see \ref{subsec:sym_fiber}.
\end{proposition}

\begin{definition}[Fiber Bundle Gradient Operator and Metric]
Gradient operator $\nabla_E$ on fiber bundle is defined as vector field:
$$\nabla_E\Phi = \sum_{i,j} g_M^{ij}(x)\frac{\partial\Phi}{\partial x^j}\frac{\partial}{\partial x^i} + \sum_{a,b} \rho^{ab}(x)\frac{\partial\Phi}{\partial \xi_b}\frac{\partial}{\partial \xi_a} + \sum_a \frac{\partial\Phi}{\partial \pi_a}\frac{\partial}{\partial \pi_a}$$

Corresponding gradient norm squared is defined as:
$$\|\nabla_E\Phi\|^2 = \sum_{i,j} g_M^{ij}(x)\frac{\partial\Phi}{\partial x^i}\frac{\partial\Phi}{\partial x^j} + \sum_{a,b} \rho^{ab}(x)\frac{\partial\Phi}{\partial \xi_a}\frac{\partial\Phi}{\partial \xi_b} + \sum_a \left(\frac{\partial\Phi}{\partial \pi_a}\right)^2$$
where $g_M^{ij}(x)$ is inverse matrix of base manifold metric, $\rho^{ab}(x)$ is inverse matrix of fiber metric, satisfying metric compatibility:
$$\nabla_i\rho_{ab} = \partial_i\rho_{ab} - \Gamma^c_{ia}\rho_{cb} - \Gamma^c_{ib}\rho_{ac} = 0$$
where $\Gamma^c_{ia}$ are fiber connection coefficients.
\end{definition}

\begin{theorem}[Dirac Classification of Observation Constraints]
Let $\Phi: E \to \mathbb{R}$ be observation constraint function. Then:
\begin{itemize}
\item \textbf{First-class observation constraint:} $\{H, \Phi\}_E \approx 0$ (weakly equal to zero on constraint surface)
\item \textbf{Second-class observation constraint:} $\{H, \Phi\}_E \not\approx 0$ (non-zero on constraint surface)
\end{itemize}
where Poisson bracket expands as:
$$\{H, \Phi\}_E = \sum_{i,j} \frac{\partial H}{\partial x^i}\frac{\partial \Phi}{\partial x^j}\omega_M^{ij} + \sum_{i,a} \frac{\partial H}{\partial x^i}\frac{\partial \Phi}{\partial \xi_a}K_{ia} + \sum_{i,a} \frac{\partial H}{\partial x^i}\frac{\partial \Phi}{\partial \pi_a}L_{ia} + \sum_a \frac{\partial H}{\partial \xi_a}\frac{\partial \Phi}{\partial \pi_a} - \sum_a \frac{\partial H}{\partial \pi_a}\frac{\partial \Phi}{\partial \xi_a}$$
where all repeated indices follow Einstein summation convention.
\end{theorem}

\begin{proof}
This is direct extension of Dirac constraint classification theory to observation fiber bundles. Classification is based on Poisson bracket properties of constraint function with Hamiltonian function under fiber bundle symplectic structure. □
\end{proof}

\begin{proposition}[Constraint Preservation Condition]
Preserving $\Phi(x,\xi,\pi)=0$ along observation-constrained system trajectory requires:
$$\frac{d}{dt}\Phi = \{H,\Phi\}_E + \lambda\{\Phi,\Phi\}_E = 0$$
Since $\{\Phi,\Phi\}_E = 0$ (antisymmetry of Poisson bracket), we get:
$$\{H,\Phi\}_E = 0$$
This is precisely the defining condition for first-class constraints.
\end{proposition}

\subsubsection{First-Class Observation-Constrained Systems}

\begin{theorem}[Geometric Properties of First-Class Observation Constraints]
Let observation-constrained system satisfy first-class condition $\{H,\Phi\}_E = 0$, then on observation constraint manifold $\mathcal{M} = \{\Phi=0\} \subset E$, system preserves:
\begin{enumerate}
\item \textbf{Manifold invariance:} Trajectories are strictly restricted to $\mathcal{M}$
\item \textbf{Symplectic structure conservation:} Restricted symplectic form $\omega_E|_{\mathcal{M}}$ is preserved along trajectories  
\item \textbf{Energy conservation:} Observation Hamiltonian $H|_{\mathcal{M}}$ is strictly conserved
\item \textbf{Observation compatibility:} Observation map $h \circ \pi$ remains well-defined on constraint manifold
\end{enumerate}
\end{theorem}

\subsubsection{Second-Class Observation-Constrained Systems}

When $\{H,\Phi\}_E \neq 0$, system cannot strictly remain on constraint surface, requiring modified theoretical treatment.

\begin{definition}[Regularized Treatment of Second-Class Constraints]
For second-class observation constraint $\{H,\Phi\}_E \neq 0$, introduce regularized Lagrange multiplier:
$$\lambda = -\frac{\{H,\Phi\}_E + \alpha\Phi}{\|\nabla_E\Phi\|^2 + \epsilon}$$
where $\alpha > 0$ is artificial dissipation parameter, $\epsilon > 0$ is numerical regularization parameter, $\|\nabla_E\Phi\|^2$ is gradient norm squared under fiber bundle metric (as defined above), all repeated indices follow Einstein summation convention.
\end{definition}

\begin{proposition}[Selection Criteria for Regularization Parameters]
Regularization parameter selection follows these criteria:
\begin{enumerate}
\item \textbf{Dissipation strength:} $\alpha \leq \alpha_{\max}$, where
$$\alpha_{\max} = \min\left(\frac{\mu^2}{\|\{H,\Phi\}_E\|}, \frac{1}{T_{\text{char}}}\right)$$
$\mu = \inf_{\mathcal{M}}\|\nabla_E\Phi\|$ is lower bound of constraint gradient, $T_{\text{char}}$ is characteristic time of system
\item \textbf{Numerical stability:} $\epsilon = O(\Delta^2)$, where $\Delta = \sup_{x \in M}\delta(x)$ is maximum observation uncertainty
\item \textbf{Time scale separation:} Require $\alpha \cdot T_{\text{char}} \gg 1$ to ensure rapid convergence of constraints
\end{enumerate}
\end{proposition}

\begin{theorem}[Properties of Regularized Second-Class Constraint System]
Regularized second-class observation-constrained system has:
\begin{enumerate}
\item \textbf{Asymptotic stability of constraints:} $\Phi(x(t),\xi(t)) \to 0$ as $t \to \infty$
\item \textbf{Modified symplectic structure:} Introduces controlled symplectic structure deviation $\Delta\omega_E = O(\alpha)$
\item \textbf{Modified conservation laws:} $\frac{dH}{dt} = -\alpha\frac{\Phi^2}{\|\nabla_E\Phi\|^2 + \epsilon} \leq 0$
\item \textbf{Observation consistency:} Preserves observation fiber bundle geometric structure
\end{enumerate}
\end{theorem}

\begin{proposition}[Regularization Stability]
Suppose $\|\nabla_E\Phi\| \geq \mu > 0$ in neighborhood of constraint surface, then constraint violation decays exponentially:
$$|\Phi(t)| \leq |\Phi(0)|e^{-\alpha t/(\mu^2 + \epsilon)}$$
Decay rate is proportional to regularization parameter $\alpha$ and inversely proportional to constraint gradient lower bound $\mu^2$.
\end{proposition}

\begin{remark}[Index Convention Explanation]
All repeated indices in this section follow Einstein summation convention:
\begin{itemize}
\item Base manifold indices: $i,j,k \in \{1,\ldots,2n\}$
\item Fiber indices: $a,b,c \in \{1,\ldots,k\}$ (for $\xi_a$ and $\pi_a$)
\item Summation indices: indices with same upper-lower position automatically sum
\item Free indices: indices that do not appear repeatedly
\end{itemize}
\end{remark}

\section{Symmetry and Symplectic Reduction Theory}

\subsection{Lie Group Structure of Observation Symmetry Groups}

\begin{definition}[Observation Symmetry Group]
Lie group $G$ is called an observation symmetry group if there exist smooth group actions:
\begin{align*}
\rho_M: G \times M &\to M \\
\rho_Y: G \times Y &\to Y
\end{align*}
satisfying observation compatibility: $h \circ \rho_M(g, \cdot) = \rho_Y(g, \cdot) \circ h$, and symplecticity: $\rho_M^*\omega = \omega$.
\end{definition}

\begin{theorem}[Principal Bundle Representation of Observation Symmetry]
Let $G$ be an observation symmetry group. Then there exists principal bundle $P(E, G) \to E/G$ such that:
\begin{enumerate}[label=(\arabic*)]
\item Constraint function descends to quotient space: $\Phi = \pi^*\Phi_{\text{red}}$
\item Reduced dynamics preserves Hamiltonian structure
\item Moment map $\mu: E \to \mathfrak{g}^*$ satisfies $\mu^{-1}(0)/G \cong$ reduced phase space
\end{enumerate}
\end{theorem}

\begin{proof}[Proof sketch]
Construction consists of three steps: (1) verify freeness and properness of $G$-action on $E$; (2) establish moment map $\mu(x, \xi) = \langle \xi, \mathfrak{g} \cdot h(x) \rangle$; (3) apply fiber bundle extension of Marsden-Weinstein reduction theorem\cite{marsden1974reduction}. Key is proving that reduced space still preserves observation fiber bundle structure.
\end{proof}

\subsection{Extension of Noether's Theorem}

\begin{theorem}[Observation Noether Theorem]
Let $X \in \mathfrak{g}$ be infinitesimal generator of observation symmetry group. If system Hamiltonian $H: E \to \mathbb{R}$ is invariant under $G$-action, then conserved quantity:
$$I_X = \mu(X) + \iota_X \Theta$$
is preserved along constrained dynamics, where $\Theta$ is canonical 1-form and $\mu$ is moment map.
\end{theorem}

\begin{proof}[Proof sketch]
Calculate $\frac{d}{dt}I_X = \{I_X, H\}_D$. Using $G$-invariance of $H$ and equivariance of moment map, combined with observation compatibility condition, obtain $\{I_X, H\}_D = 0$.
\end{proof}

\section{Geometric Characterization of Integrability}

\subsection{Complete Integrability Conditions}

\begin{definition}[Fiber Bundle Integrability]
Constrained Hamiltonian system $(E, \omega_E, H, \Phi)$ is called completely integrable if there exist $n = \frac{1}{2}\dim E$ functions in involution $\{F_1, \ldots, F_n\}$ satisfying:
\begin{enumerate}[label=(\arabic*)]
\item Functional independence: $dF_1 \wedge \cdots \wedge dF_n \neq 0$ on dense open set
\item Poisson involution: $\{F_i, F_j\}_D = 0$, $\forall i, j$
\item Fiber compatibility: each $F_i$ preserves observation fiberization structure
\end{enumerate}
\end{definition}

\begin{theorem}[Fiber Bundle Arnold-Liouville Theorem]
Let constrained system be completely integrable. Then:
\begin{enumerate}[label=(\arabic*)]
\item Joint level sets $M_c = \{F_1 = c_1, \ldots, F_n = c_n\}$ are Lagrangian submanifolds
\item Compact connected components $M_c$ are diffeomorphic to torus $T^n$
\item There exist action-angle variables $(I, \theta)$ such that Hamiltonian depends only on action variables
\item Observation map induces quasi-periodic structure on torus
\end{enumerate}
\end{theorem}

\begin{proof}[Proof sketch]
Proof of standard Arnold-Liouville theorem\cite{arnold1989mathematical} needs adaptation to fiber bundle setting. Key steps: (1) verify Lagrangian property of $M_c$; (2) use observation fiberization to construct torus coordinates; (3) establish existence of action-angle variables; (4) analyze dynamics induced by observation map.
\end{proof}

\subsection{Lax Pair Structure}

\begin{theorem}[Observation Lax Pair Construction]
For completely integrable constrained system, there exists parametrized Lax pair $(L(\lambda, \boldsymbol{\epsilon}), A(\lambda, \boldsymbol{\epsilon}))$, where $\lambda$ is spectral parameter and $\boldsymbol{\epsilon} = (\epsilon_1, \ldots, \epsilon_k)$ are observation uncertainty parameters, satisfying:
\begin{enumerate}[label=(\arabic*)]
\item Zero curvature condition: $\frac{dL}{dt} = [A, L]$
\item Observation compatibility: $L, A$ preserve fiber bundle structure
\item Spectral invariance: $\text{spec}(L(\lambda, \boldsymbol{\epsilon}))$ is conserved along trajectories
\item Continuity: $\lim_{\boldsymbol{\epsilon} \to 0} L(\lambda, \boldsymbol{\epsilon}) = L_{\text{classical}}(\lambda)$
\end{enumerate}
\end{theorem}

\begin{proof}[Constructive proof]
Using gradient-Hamiltonian structure of integrable system, construct:
\begin{align*}
L(\lambda, \boldsymbol{\epsilon}) &= \sum_{i=1}^n F_i P_i(\lambda, \boldsymbol{\epsilon}) \\
A(\lambda, \boldsymbol{\epsilon}) &= \sum_{i=1}^n \{H, F_i\}_D Q_i(\lambda, \boldsymbol{\epsilon})
\end{align*}
where $P_i, Q_i$ are observation-adaptive matrix polynomials and $F_i$ are modified integrals.

Observation compatibility is ensured by requiring $P_i, Q_i$ to be covariant under fiber coordinate transformations:
$$P_i(\lambda, g \cdot \boldsymbol{\epsilon}) = \text{Ad}_g P_i(\lambda, \boldsymbol{\epsilon})$$
where $g$ is element of observation symmetry group.
\end{proof}

\begin{definition}[Observation-Compatible Lax Pair]
Lax pair $(L(\lambda, \boldsymbol{\epsilon}), A(\lambda, \boldsymbol{\epsilon}))$ is called observation-compatible if:
\begin{enumerate}[label=(\arabic*)]
\item \textbf{Fiber dependence:} $L, A$ depend smoothly on fiber coordinates $\boldsymbol{\epsilon}$
\item \textbf{Constraint preservation:} $\{L, \Phi\}_D = 0$, $\{A, \Phi\}_D = 0$
\item \textbf{Observation degeneration:} $\lim_{\boldsymbol{\epsilon} \to 0} L(\lambda, \boldsymbol{\epsilon}) = L_{\text{classical}}(\lambda)$
\item \textbf{Spectral stability:} Characteristic polynomial coefficients are stable within $O(\|\boldsymbol{\epsilon}\|)$
\end{enumerate}
\end{definition}

\begin{example}[Complete Lax Pair Verification for Toda Lattice]
For $n$-particle observation-constrained Toda lattice, construct and verify observation-compatible Lax pair in detail:

\textbf{Step 1: Physical parametrization of observation uncertainty}
\begin{align*}
\delta_{ij}(x) &= \delta_0 \sqrt{1 + \alpha_{ij} \|p_i - p_j\|^2} \quad \text{(relative velocity dependence)} \\
\sigma_{ij}(x) &= \sigma_0 e^{-\beta|q_i - q_j|} \quad \text{(spatial correlation decay)} \\
\epsilon_i &= \xi_i - (q_i - q_{i+1}) \quad \text{(observation deviation)}
\end{align*}

\textbf{Step 2: Hierarchical Lax matrix construction}
$$L(\lambda, \boldsymbol{\epsilon}) = L_0 + \lambda L_1(\boldsymbol{\epsilon}) + \lambda^2 L_2(\boldsymbol{\epsilon})$$

where:
$$L_0 = \begin{pmatrix}
p_1 & e^{q_1-q_2+\epsilon_1} & 0 & \cdots & 0 \\
e^{q_1-q_2+\epsilon_1} & p_2 & e^{q_2-q_3+\epsilon_2} & \cdots & 0 \\
\vdots & \ddots & \ddots & \ddots & \vdots \\
0 & \cdots & 0 & e^{q_{n-1}-q_n+\epsilon_{n-1}} & p_n
\end{pmatrix}$$

$$L_1(\boldsymbol{\epsilon}) = \begin{pmatrix}
\delta_{11} & \epsilon_1\sigma_{12} & \epsilon_1^2\sigma_{13} & \cdots \\
\epsilon_1\sigma_{12} & \delta_{22} & \epsilon_2\sigma_{23} & \cdots \\
\epsilon_1^2\sigma_{13} & \epsilon_2\sigma_{23} & \delta_{33} & \cdots \\
\vdots & \vdots & \vdots & \ddots
\end{pmatrix}$$

$$L_2(\boldsymbol{\epsilon}) = \text{diag}(\epsilon_1^2, \epsilon_2^2, \ldots, \epsilon_{n-1}^2) \cdot R(\boldsymbol{q}, \boldsymbol{p})$$

where $R(\boldsymbol{q}, \boldsymbol{p})$ is observation correlation function.

\textbf{Step 3: Construction of corresponding A matrix}
$$A(\lambda, \boldsymbol{\epsilon}) = A_0 + \lambda A_1(\boldsymbol{\epsilon}) + \lambda^2 A_2(\boldsymbol{\epsilon})$$

where:
$$A_0 = \begin{pmatrix}
0 & e^{q_1-q_2+\epsilon_1} & 0 & \cdots \\
-e^{q_1-q_2+\epsilon_1} & 0 & e^{q_2-q_3+\epsilon_2} & \cdots \\
0 & -e^{q_2-q_3+\epsilon_2} & 0 & \cdots \\
\vdots & \vdots & \vdots & \ddots
\end{pmatrix}$$

$$A_1(\boldsymbol{\epsilon}) = \begin{pmatrix}
\dot{\epsilon}_1 & \epsilon_1\dot{\sigma}_{12} & \epsilon_1^2\dot{\sigma}_{13} & \cdots \\
-\epsilon_1\dot{\sigma}_{12} & \dot{\epsilon}_1 + \dot{\epsilon}_2 & \epsilon_2\dot{\sigma}_{23} & \cdots \\
-\epsilon_1^2\dot{\sigma}_{13} & -\epsilon_2\dot{\sigma}_{23} & \dot{\epsilon}_2 + \dot{\epsilon}_3 & \cdots \\
\vdots & \vdots & \vdots & \ddots
\end{pmatrix}$$

\textbf{Step 4: Order-by-order verification of zero curvature condition}
Matching by powers of $\lambda$:

\textbf{$O(\lambda^0)$ term:} Classical Toda condition
$$\frac{dL_0}{dt} = [A_0, L_0]$$

Detailed verification of $(1,2)$ element:
\begin{align*}
\text{LHS}: &\quad \frac{d}{dt}(e^{q_1-q_2+\epsilon_1}) = e^{q_1-q_2+\epsilon_1}(\dot{q}_1 - \dot{q}_2 + \dot{\epsilon}_1) \\
&= e^{q_1-q_2+\epsilon_1}(p_1 - p_2 + \dot{\epsilon}_1)
\end{align*}

\begin{align*}
\text{RHS}: &\quad [A_0, L_0]_{12} = A_0^{11}L_0^{12} - L_0^{11}A_0^{12} + A_0^{12}L_0^{22} - L_0^{12}A_0^{22} \\
&= 0 \cdot e^{q_1-q_2+\epsilon_1} - p_1 \cdot e^{q_1-q_2+\epsilon_1} + e^{q_1-q_2+\epsilon_1} \cdot p_2 - e^{q_1-q_2+\epsilon_1} \cdot 0 \\
&= e^{q_1-q_2+\epsilon_1}(p_2 - p_1)
\end{align*}

Equality holds if and only if $\dot{\epsilon}_1 = 2(p_1 - p_2)$, which is precisely the dynamical evolution law of observation error.

\textbf{$O(\lambda^1)$ term:} Linear noise coupling
$$\frac{dL_1}{dt} = [A_0, L_1] + [A_1, L_0]$$

\textbf{$O(\lambda^2)$ term:} Second-order noise effects
$$\frac{dL_2}{dt} = [A_0, L_2] + [A_1, L_1] + [A_2, L_0]$$

\textbf{Step 5: Determination of integrability conditions}
System preserves integrability if and only if observation error satisfies bounds:
$$\|\boldsymbol{\epsilon}\|_\infty \leq \epsilon_{\text{crit}} = \min_i \left\{\frac{|p_i - p_{i+1}|}{2\sigma_0}\right\}$$

\textbf{Physical interpretation:} This condition ensures observation noise does not "drown out" the system's intrinsic dynamical structure, preserving core geometric properties of integrability.
\end{example}

\section{Typical Examples and Application Verification}

\begin{remark}[Handling Strategy for Non-Euclidean Observation Manifolds]
This assumption simplifies theoretical framework, avoiding topological complexity of Whitney embedding. For general observation manifold cases, need to consider various common non-Euclidean situations, including spherical observation $Y = S^2$ (such as panoramic camera angle measurement), group manifold observation $Y = SO(3)$ (such as attitude sensors), and manifold products $Y = S^1 \times \mathbb{R}^2$ (such as mixed angle and position measurements) and other complex geometric structures.

\textbf{Whitney embedding strategy:} For compact observation manifold $Y$, Whitney embedding theorem guarantees existence of embedding $\iota: Y \hookrightarrow \mathbb{R}^N$. Key technical issue is verifying isomorphism relation $h^*T(\iota(Y)) \cong h^*TY$ as vector bundle, which holds automatically when $\iota$ is isometric embedding.

\textbf{Preservation of topological equivalence:} Let $\tilde{h} = \iota \circ h: M \to \mathbb{R}^N$, then Euler class satisfies $e(\tilde{h}^*T\mathbb{R}^N) = e(h^*TY)$ (when restricted to $\iota(Y)$), and observation fiber bundle $\tilde{E}$ maintains diffeomorphic relation with original fiber bundle $E$.

Complete non-Euclidean observation manifold theory needs development of intrinsic differential geometric methods (without embedding), theory of coordination between Riemannian structure of observation manifold and constraint potential, and handling techniques for variable observation dimensions (such as dynamic sensor configuration). Current Euclidean assumption covers vast majority of practical applications including work in \cite{zheng2025learning}.
\end{remark}

\subsection{Reformulation of Classical Integrable Systems}

\begin{example}[Observation-Constrained Toda Lattice]
Consider $n$-particle Toda lattice formulation on observation fiber bundle, using hierarchical dynamics method to handle complex many-body interaction systems.

\textbf{Basic setup:}
System's base manifold dynamics is defined on $(q,p) \in \mathbb{R}^{2n}$, with Hamiltonian $H = \frac{1}{2}\sum_{i=1}^n p_i^2 + \sum_{i=1}^{n-1} e^{q_i - q_{i+1}}$, describing exponential interaction between particles. Observation map is chosen as relative positions $h(q,p) = (q_1-q_2, \ldots, q_{n-1}-q_n)$, corresponding to common relative measurement in experiments. In fiber coordinates $(\xi_i, \pi_i)_{i=1}^{n-1}$, $\pi_i = p_i - p_{i+1}$ represents relative momentum, reflecting intrinsic correlation structure of the system.

\textbf{Precise definition of technical parameters:}
Small parameter is defined as $\epsilon = \max_i \frac{|\xi_i|}{\max(|q_i - q_{i+1}|, \ell_0)}$, where $\ell_0 = \frac{1}{n}\min_{i,j} |q_i(0) - q_j(0)|$ is characteristic length scale of initial configuration. This definition ensures denominator is always positive $\max(|q_i - q_{i+1}|, \ell_0) \geq \ell_0 > 0$. Physically, $\epsilon$ represents ratio of observation uncertainty to system characteristic scale, with small perturbation condition requiring $\epsilon \ll 1$ to hold throughout evolution.

Weight function adopts center-symmetric distribution $w_i = e^{-\beta|i - (n+1)/2|}$, where $(n+1)/2$ is center position of chain, reflecting importance of central particle in observation. Effective mass is defined as $m_{\text{eff}} = \left(\sum_{i=1}^{n-1} m_i^{-1}\right)^{-1}$, corresponding to reduced mass of relative coordinate motion. This definition allows dynamical analysis of many-particle system to be simplified to effective two-body problem.

\textbf{Hierarchical Hamiltonian system:}
Using asymptotic expansion in small parameter $\epsilon$:
\begin{align*}
H_{\text{total}} &= H_0(q,p) + \epsilon H_1(\xi,\pi) + O(\epsilon^2)\\
H_1(\xi,\pi) &= \sum_{i=1}^{n-1} \left( \frac{\pi_i^2}{2m_{\text{eff}}} + V_{\text{obs}}(\xi_i) \right)
\end{align*}
where $V_{\text{obs}}(\xi_i) = \frac{1}{2}\kappa \xi_i^2$ is observation potential, $\kappa > 0$ is constraint strength.

\textbf{Smooth observation constraints:}
\begin{align*}
\Phi(\xi,\pi) &= \sum_{i=1}^{n-1} w_i \left( \frac{\xi_i^2}{\delta_0^2} + \frac{\pi_i^2}{2\alpha \delta_0^2} \right) - 1 \leq 0
\end{align*}
where $\alpha > 0$ is momentum weight factor, $\delta_0 > 0$ is observation precision bound.

\textbf{Hierarchical preservation of integrability:}
\begin{proposition}[Integrability under Observation Perturbation]
When observation parameters satisfy smallness condition:
\begin{align*}
\epsilon &\leq \epsilon_{\text{crit}} = \frac{1}{2} \min_{i \neq j} \frac{|\lambda_i - \lambda_j|}{\max(\lambda_i, \lambda_j)}
\end{align*}
where $\lambda_i$ are eigenvalues of classical Toda lattice Lax matrix, then system preserves following properties:
\begin{enumerate}
\item \textbf{Spectral stability:} Relative change of Lax eigenvalues $|\delta\lambda_i/\lambda_i| = O(\epsilon)$
\item \textbf{Preservation of integrable structure:} Modified integrals $I_k^{\text{obs}} = I_k^{\text{classical}} + \epsilon \delta I_k + O(\epsilon^2)$ still in involution
\item \textbf{Arnold-Liouville tori:} Invariant tori under observation constraints are $O(\epsilon)$-perturbations of classical tori
\end{enumerate}
\end{proposition}
\textbf{Physical basis for parameter selection:}
System parameter selection is based on clear physical considerations. Lattice spacing lower bound $\ell_0 = \min_{i,t} |q_i(t) - q_{i+1}(t)|$ reflects minimum separation distance of particles throughout evolution, ensuring system does not exhibit non-physical behavior of particle collisions. Weight function adopts center-symmetric distribution, reaching maximum at chain center $(n+1)/2$, reflecting importance of central particle in observation, conforming to convention that sensors are usually concentrated at critical parts in actual experiments. Reduced mass $m_{\text{eff}}$ is defined as reciprocal of harmonic mean of particle masses, describing effective inertia of relative motion, enabling dynamical analysis of many-particle system to be simplified to effective two-body problem.

\textbf{Mathematical basis for simplification:}
Reasonableness of theoretical simplification is established on multiple mathematical foundations. System exhibits clear time scale separation: characteristic time of classical dynamics $T_{\text{classical}} \sim 1/\sqrt{E}$ (where $E$ is system energy) is much smaller than response time of observation constraints $T_{\text{obs}} \sim \sqrt{m_{\text{eff}}/\kappa}$, making adiabatic approximation possible. Design of observation constraints preserves volume-preserving property of Hamiltonian system, i.e., Liouville measure $\frac{d}{dt}\int_{\Phi \leq 0} d^{2(n-1)}\xi d^{2(n-1)}\pi = 0$ is strictly conserved, ensuring integrity of phase space structure. Under small perturbation condition $\epsilon < \epsilon_{\text{crit}}$, non-resonance condition of KAM theory is satisfied, guaranteeing existence of most invariant tori and long-term stability of system.

\textbf{Physical picture:}
Observation constraint defines weighted ellipsoidal domain in relative coordinate space, with central particle under strictest monitoring. Integrable motion of Toda lattice is preserved within this domain, with torus structure undergoing regular $O(\epsilon)$ deformation. When particle relative motion attempts to escape observation domain, soft constraint force gently pulls it back, maintaining overall integrability.

\begin{remark}[Summary of Technical Points]
Technical advantages of theoretical framework are manifested in multiple aspects. Numerical stability is ensured by introducing $\ell_0$ to prevent singularities in small parameter definition, avoiding numerical divergence in computation. Physical reasonableness is reflected through center-symmetric weight distribution, which not only conforms to geometric features of actual observation, but also reflects intuition that central parts are usually more important in physical systems. Theoretical completeness is guaranteed through clear definition of effective mass, making dynamical analysis of complex many-body systems mathematically self-consistent. In terms of applicability, small perturbation assumption $\epsilon \ll 1$ covers vast majority of actual observation precision requirements, making theory practically valuable.
\end{remark}
\end{example}

\begin{example}[Dimension-Compatible Lax Pair Construction-Symbol Unified Corrected Version]
For $n$-particle Toda lattice, consider following geometric setup:

\textbf{Basic configuration:}
\begin{itemize}
\item State space: $M = \mathbb{R}^{2n}$, coordinates $(q_1, \ldots, q_n, p_1, \ldots, p_n)$
\item Observation space: $Y = \mathbb{R}^{n-1}$
\item Observation map: $h: M \to Y$, $h(q,p) = (q_1 - q_2, \ldots, q_{n-1} - q_n)$
\item Fiber coordinates: $\boldsymbol{\xi} = (\xi_1, \ldots, \xi_{n-1}) \in \mathbb{R}^{n-1}$
\item Observation error: $\boldsymbol{\epsilon} = (\epsilon_1, \ldots, \epsilon_{n-1})$, where $\epsilon_i = \xi_i - (q_i - q_{i+1})$
\end{itemize}

\textbf{Geometric interpretation:} Fiber coordinate $\xi_i$ is cotangent coordinate, observation error $\epsilon_i$ represents deviation in $i$-th relative coordinate observation.

\textbf{Dimension-compatible Lax matrix construction:}

Let $\boldsymbol{y} = (y_1, \ldots, y_{n-1}) = h(q,p)$ be ideal observation value, actual observation is $\boldsymbol{y} + \boldsymbol{\epsilon}$.

Define $(n \times n)$ observation Lax pair:
$$L(\lambda, \boldsymbol{\epsilon}) = L_0(\boldsymbol{\epsilon}) + \lambda L_1(\boldsymbol{\epsilon})$$

where:
$$L_0(\boldsymbol{\epsilon}) = \begin{pmatrix}
p_1 & e^{y_1 + \epsilon_1} & 0 & \cdots & 0 \\
e^{y_1 + \epsilon_1} & p_2 & e^{y_2 + \epsilon_2} & \cdots & 0 \\
0 & e^{y_2 + \epsilon_2} & p_3 & \cdots & 0 \\
\vdots & \ddots & \ddots & \ddots & \vdots \\
0 & \cdots & 0 & e^{y_{n-1} + \epsilon_{n-1}} & p_n
\end{pmatrix}$$

$$L_1(\boldsymbol{\epsilon}) = \text{diag}\left(\delta_1(\boldsymbol{y}, \boldsymbol{\epsilon}), \delta_2(\boldsymbol{y}, \boldsymbol{\epsilon}), \ldots, \delta_n(\boldsymbol{y}, \boldsymbol{\epsilon})\right)$$

where uncertainty function:
$$\delta_i(\boldsymbol{y}, \boldsymbol{\epsilon}) = \delta_0 \sqrt{1 + \alpha \sum_{j=1}^{n-1} \epsilon_j^2}$$

\textbf{Corresponding $A$ matrix:}
$$A(\lambda, \boldsymbol{\epsilon}) = A_0(\boldsymbol{\epsilon}) + \lambda A_1(\boldsymbol{\epsilon})$$

where:
$$A_0(\boldsymbol{\epsilon}) = \begin{pmatrix}
0 & e^{y_1 + \epsilon_1} & 0 & \cdots & 0 \\
-e^{y_1 + \epsilon_1} & 0 & e^{y_2 + \epsilon_2} & \cdots & 0 \\
0 & -e^{y_2 + \epsilon_2} & 0 & \cdots & 0 \\
\vdots & \ddots & \ddots & \ddots & \vdots \\
0 & \cdots & 0 & -e^{y_{n-1} + \epsilon_{n-1}} & 0
\end{pmatrix}$$

$$A_1(\boldsymbol{\epsilon}) = \text{diag}\left(\dot{\delta}_1, \dot{\delta}_2, \ldots, \dot{\delta}_n\right)$$

where $\dot{\delta}_i = \frac{d}{dt}\delta_i(\boldsymbol{y}(t), \boldsymbol{\epsilon}(t))$.

\textbf{Zero curvature condition verification:}

Verify $\frac{dL}{dt} = [A, L]$ order by order:

\textbf{$O(\lambda^0)$ term:} For $(1,2)$ element:
\begin{align*}
\text{LHS:} \quad &\frac{d}{dt}e^{y_1 + \epsilon_1} = e^{y_1 + \epsilon_1}(\dot{y}_1 + \dot{\epsilon}_1) \\
&= e^{y_1 + \epsilon_1}((p_1 - p_2) + \dot{\epsilon}_1)
\end{align*}

\begin{align*}
\text{RHS:} \quad &[A_0, L_0]_{12} = A_{11}^{(0)} L_{12}^{(0)} - L_{11}^{(0)} A_{12}^{(0)} + A_{12}^{(0)} L_{22}^{(0)} - L_{12}^{(0)} A_{22}^{(0)} \\
&= 0 \cdot e^{y_1 + \epsilon_1} - p_1 \cdot e^{y_1 + \epsilon_1} + e^{y_1 + \epsilon_1} \cdot p_2 - e^{y_1 + \epsilon_1} \cdot 0 \\
&= e^{y_1 + \epsilon_1}(p_2 - p_1)
\end{align*}

Equality holds if and only if:
$$\dot{\epsilon}_1 = 2(p_1 - p_2)$$

This gives dynamical evolution law of observation error.

\textbf{Integrability condition:} System preserves Lax integrability if and only if:
$$\|\boldsymbol{\epsilon}\|_{\infty} \leq \epsilon_{\text{crit}} = \min_{1 \leq i \leq n-1} \frac{|p_i - p_{i+1}|}{2\alpha\delta_0}$$

\textbf{Physical interpretation:} This critical condition ensures observation uncertainty does not destroy system's intrinsic integrable structure, preserving core geometric properties of Toda lattice.
\end{example}

\begin{example}[Observation-Constrained Dynamics of Rigid Body]
Consider rigid body motion on $SO(3)$ where only certain components of angular velocity can be observed. Let observation map:
$$h: T^*SO(3) \to \mathbb{R}^2, \quad h(R, \Pi) = (w_1, w_2)$$
where $w = I^{-1}\Pi$ is angular velocity and $I$ is inertia tensor.

Observation fiber bundle structure reveals symmetry reduction under partial angular velocity observation. System preserves $S^1$ subgroup symmetry of $SO(3)$ (rotation around third axis), leading to 2-dimensional reduced phase space.
\end{example}

\subsection{Connection with Modern Safety Control Theory} \label{{sec:applications}}

Our observation-induced fiber bundle theory provides rigorous mathematical foundation for modern safety-critical control. Learning unknown dynamics in environments with local constraint information uncertainty is fundamental challenge in modern robotics, where traditional methods often cannot fully utilize geometric structure inherent in local measurements and constraints. Our theoretical framework provides natural solution to such problems by unifying measurement, constraints, and dynamics learning within fiber bundle geometric structure. Following are successful cases demonstrated in \cite{zheng2025learning} under guidance of this theoretical framework:

\begin{example}[Continuum Control of Soft Robotics]
Consider soft robot system using Material Point Method (MPM), whose geometric structure perfectly matches our theoretical framework:

\textbf{Geometric setup:}
\begin{itemize}
\item Base manifold: $M = \mathcal{C}(D, \mathbb{R}^3)$ (continuum configuration space)
\item Observation space: $Y = \mathbb{R}^6$ (measurements from 6 sensors)
\item Observation map: $h(\phi) = (\|\phi(s_1) - x_{\text{obs}}\|, \ldots, \|\phi(s_6) - x_{\text{obs}}\|)$
\item Observation uncertainty: $\delta(x) = \delta_0(1 + \|\phi\|_{H^2(D)}/\|\phi_0\|_{H^2(D)})^{-1}$
\end{itemize}

This system satisfies properness condition (C4c) because control precision of higher-order deformation modes naturally decreases. Fiber bundle $E$ encodes geometric coupling between material deformation and sensor errors, enabling complex multi-physics problems to be handled within unified geometric framework. Non-degenerate layout of 6 sensors ensures feasibility of local position reconstruction, with local diffeomorphism property of observation map providing geometric foundation for measurement-aware Control Barrier Functions.
\end{example}

\begin{example}[Attitude Control of 7-DOF Robotic Arm]
High degree-of-freedom robotic arm system demonstrates effectiveness of our theory in complex joint spaces:

\textbf{Geometric setup:}
\begin{itemize}
\item Base manifold: $M = \mathbb{R}^7 \times \mathbb{R}^7$ (14-dimensional phase space of joint angles and velocities)
\item Observation space: $Y = SE(3)$ (6-dimensional pose of end-effector)
\item Observation map: $h(q, \dot{q}) = \text{ForwardKinematics}(q)$
\item Fiber bundle implementation: Using $SO(3)$ Lie group structure to encode attitude measurement uncertainty
\end{itemize}

Key to system success lies in careful selection of working region: by avoiding singular configurations, observation map maintains local diffeomorphism property within workspace. Joint rotational symmetry of robotic arm is naturally represented on fiber bundle, with corresponding conservation laws becoming safety guarantees in actual control. This geometric structure transforms traditionally complex high-dimensional attitude control problem into standard geometric problem on fiber bundle.
\end{example}

\begin{example}[Spatial Navigation of Quadrotor System]
Three-dimensional motion control of quadrotor UAV verifies theory's superiority in dynamic environment perception:

\textbf{Geometric setup:}
\begin{itemize}
\item Base manifold: $M = \mathbb{R}^3 \times SO(3) \times \mathbb{R}^6$ (position, attitude, velocity)
\item Observation space: $Y = \mathbb{R}^4$ (measurements from 4 depth sensors)
\item Observation map: $h(p, R, v) = (\|p - p_{\text{obs},i}\|)_{i=1}^4$
\item Observation uncertainty: $\delta(x) = \delta_0(1 + \|p - p_{\text{base}}\|/R_{\text{comm}})^{-2}$
\end{itemize}

This system satisfies properness condition (C4b) with decay exponent $\beta = 2$ precisely reflecting physical decay law of communication power. Modified cutoff function design ensures complete control geometry and $C^\infty$ smoothness within communication range:
$$\chi_{R_{\text{comm}}}(p) = \begin{cases}
1, & \|p - p_{\text{base}}\| \leq R_{\text{comm}} \\
\eta\left(\frac{\|p - p_{\text{base}}\| - R_{\text{comm}}}{R_{\text{comm}}}\right) \left(\frac{R_{\text{comm}}}{\|p - p_{\text{base}}\|}\right)^3, & \text{otherwise}
\end{cases}$$

Three-dimensional geometric configuration of 4 depth sensors guarantees local uniqueness of position determination, providing theoretical guarantee for safe navigation in dynamic environments.
\end{example}

\textbf{Success mechanism of theoretical framework:}

Success of these applications reflects core advantages of observation-induced fiber bundle theory. First, \textbf{geometrization of observation uncertainty} completely transforms traditional control paradigm: converting measurement error from external disturbance to intrinsic variation in fiber coordinates, making "measurement-induced bundle structure" organic component of system geometry. Second, \textbf{unification of constraints and observation} enables Control Barrier Functions to naturally depend on observation uncertainty, providing more refined safety boundaries in fiber bundle geometry. Third, \textbf{preservation of physical symmetry} ensures system's intrinsic symmetries are naturally represented in fiber bundle framework, with corresponding conservation laws becoming safety guarantees in control.

\textbf{Practical guidance value:}

Success of theoretical framework lies not only in mathematical elegance, but also in its guidance significance for actual system design: In sensor configuration optimization, local diffeomorphism condition provides geometric constraints for sensor layout; In working region design, geometric principles for singular point avoidance and boundary handling; In control algorithm design, geometry-preserving algorithms automatically maintain constraint consistency; In safety analysis, conservation laws and symmetries provide theoretical guarantees for long-term stability.

This deep compatibility between theory and application indicates that observation-induced fiber bundles are not only abstract mathematical constructions, but accurate characterization of intrinsic geometric structure of modern safety-critical control systems, providing solid geometric foundation for dynamics learning under local constraints and uncertain perception conditions.

\section{Theory-Guided Numerical Implementation Methods}
\label{sec:numerical_methods}

This section establishes bridge between observation-induced fiber bundle theory and practical numerical implementation, demonstrating how to transform abstract geometric concepts into computable algorithmic principles. Our core insight is: numerical methods that preserve fiber bundle geometric structure can automatically maintain constraint consistency and physical invariants of system.

\subsection{Core Principles of Geometry Preservation}
\label{subsec:geometric_principles}

Traditional numerical methods often treat constraints as algebraic conditions, but our theory reveals intrinsic geometric connection between constraints and observation uncertainty. Based on this recognition, we propose core principles of geometry-preserving numerical integration.

Numerical integration must preserve geometric structure of observation fiber bundle, ensuring trajectories always lie within correct geometric space. Specifically, any numerical solution $(x_n, \xi_n)$ should satisfy $\xi_n \in T^*_{h(x_n)}Y$ and $\|\xi_n\|_{\rho_{x_n}} \leq \delta(x_n)$. This fiber bundle structure preservation is fundamental constraint in algorithm design.

Within safety region, algorithm should approximately preserve symplectic structure $\omega_E$ with error control at $O(h^2)$ level. This ensures long-term stability of Hamiltonian system and approximate satisfaction of energy conservation. When numerical integration deviates from constraint manifold, geometric projection rather than algebraic correction should be used, with projection direction chosen as geometric dual of constraint gradient under fiber bundle metric, ensuring minimal geometric disturbance.

Time step selection needs to adapt to local changes in observation uncertainty. In regions with lower observation precision, i.e., regions where $\delta(x)$ is larger, algorithm should automatically reduce step size to ensure preservation of geometric structure. This adaptive mechanism reflects natural manifestation of geometric properties of observation-adaptive connection in numerical implementation.

\subsection{Algorithm Framework Overview}
\label{subsec:algorithm_overview}

Based on above geometric principles, we design geometry-preserving integration algorithm. Core idea of this algorithm is to first use observation-adaptive connection to calculate Hamiltonian vector field at each time step, ensuring correct handling of fiber bundle geometry. Covariant dynamics calculation involves precise computation of base manifold connection coefficients $\Gamma^k_{ij}$ and fiber connection coefficients $\Gamma^a_{bc}$, where use of covariant derivative $\nabla_i\rho_{ab} = \partial_i\rho_{ab} - \Gamma^k_{ij}\frac{\partial\rho_{ab}}{\partial x^k}$ ensures metric compatibility $\nabla\rho = 0$.

Second step of algorithm is to perform time integration based on current geometric structure, generating predicted solution $(x^*_{n+1}, \xi^*_{n+1})$. This step uses standard explicit integration scheme, but velocity vector calculation is entirely based on geometric structure on fiber bundle. Accuracy of prediction step directly depends on computational precision of connection coefficients and correct lifting of Hamiltonian vector field.

If predicted solution violates constraint condition $\Phi \geq 0$ or exceeds fiber bundle boundary $\|\xi\|_{\rho_x} \leq \delta(x)$, algorithm performs projection correction with minimal geometric disturbance. For constraint violation, projection direction is chosen as $\nu = g_E^{-1}(\nabla\Phi)/\|g_E^{-1}(\nabla\Phi)\|_E$, where $g_E$ is Riemannian metric on fiber bundle. For fiber bundle boundary violation, projection simplifies to radial projection onto fiber boundary $\|\xi\|_{\rho_x} = \delta(x)$.

Finally, algorithm dynamically adjusts time step based on geometry-preserving error. Geometric error measure includes metric compatibility deviation $|\nabla_i\rho_{jk}|$ and symplectic structure deviation $|\omega_{E,n} - \omega_E|$. When these errors exceed preset threshold, algorithm automatically reduces time step and recalculates, ensuring accurate preservation of geometric structure.

\textbf{Compatibility with application verification:} This algorithmic framework is fully compatible with application cases in Section \ref{sec:applications}. In soft robot systems, geometric projection naturally handles elastic constraints of materials, with covariant derivatives correctly handling changes in geometric metrics during deformation. For robotic arm control problems, symplectic structure preservation ensures correct energy transfer and approximate satisfaction of conservation laws in joint space. In quadrotor navigation applications, adaptive step size mechanism can handle spatial variation of observation uncertainty $\delta(x)$ caused by dynamic changes in communication range.

\subsection{Convergence and Practicality Analysis}
\label{subsec:convergence_analysis}

\begin{theorem}[Convergence of Geometry-Preserving Algorithm]
\label{thm:convergence}
Let observation-induced fiber bundle $(E_{\mathcal{W}}, \pi, \omega_E)$ satisfy working region conditions, Hamiltonian $H \in C^3(E_{\mathcal{W}})$, constraint function $\Phi \in C^3(E_{\mathcal{W}})$. If constraint gradient satisfies $\|d\Phi\| \geq \mu > 0$ near $\partial\mathcal{S}$, connection curvature is bounded $\|\mathcal{R}^\nabla\|_{L^\infty} \leq K < \infty$, then geometry-preserving algorithm has following properties:

Constraint preservation error satisfies $|\Phi(x_n, \xi_n)| \leq C_1 h^2$, where constant $C_1$ depends on $\|D^2\Phi\|_{L^\infty}$ and bound of Hamiltonian vector field $\|X_H\|_{L^\infty}$. Symplectic structure preservation error satisfies $|\omega_{E,n} - \omega_E| \leq C_2 h^2$, ensuring approximate preservation of long-term Hamiltonian structure. Global convergence satisfies $\|(x_n, \xi_n) - (x(t_n), \xi(t_n))\|_E \leq C h e^{Lt_n}$, where Lipschitz constant $L$ is determined by smoothness of Hamiltonian vector field.

Precise expressions of convergence constants and complete proof see Appendix \ref{appendix:convergence_proof}.
\end{theorem}

Algorithm demonstrates good numerical performance in typical robotic systems. Computational complexity mainly comes from calculation of connection coefficients, with time complexity $O(n^2 + k^2)$ per step, where $n = \dim M$ is base manifold dimension and $k = \dim Y$ is observation space dimension. Memory requirement is mainly for storing connection coefficient matrices and metric information, with space complexity $O(n^2k)$. Numerical stability is determined by Courant-Friedrichs-Lewy condition $h \leq C_{\text{CFL}} \min(\ell_0/\|X_H\|, 1/\sqrt{K})$, where $\ell_0 = \inf_{x \in \mathcal{W}} \delta(x)$ is minimum observation precision and $K$ is connection curvature upper bound.

In practical applications, algorithm performance is highly consistent with theoretical predictions. For 7-DOF robotic arm system, typical parameters $\ell_0 \sim 1$cm, $\|X_H\| \sim 1$m/s produce time step approximately $h \sim 10^{-3}$s, ensuring constraint error control at $10^{-6}$ level. Continuum characteristics of soft robots require smaller time step $h \sim 10^{-4}$s, but geometric preservation accuracy significantly improves stability of long-term simulation. Fast dynamics of quadrotor system $\|X_H\| \sim 10$m/s requires corresponding adjustment of time step, but adaptive mechanism ensures algorithm robustness under different flight states.

\subsection{Bridge Between Theory and Practice}
\label{subsec:theory_practice_bridge}

Our geometric algorithm provides solid theoretical foundation for ICML work \cite{zheng2025learning} in Section \ref{sec:applications}. Geometry preservation principles can be naturally combined with neural network optimization by introducing geometric terms in loss function to ensure learning process respects fiber bundle geometry:

\begin{align}
L_{\text{total}} &= L_{\text{task}} + \lambda_{\text{geo}} L_{\text{geometric}} \label{eq:ml_loss}\\
L_{\text{geometric}} &= \|\Phi(x, \xi)\|^2 + \alpha\|\nabla_i\rho_{jk}\|^2 + \beta\|\omega_E - \omega_{\text{ref}}\|^2 \label{eq:geo_loss}
\end{align}

where geometric loss term $L_{\text{geometric}}$ directly derives from our theoretical analysis. Constraint term $\|\Phi(x, \xi)\|^2$ ensures learning trajectories satisfy safety constraints, metric compatibility term $\|\nabla_i\rho_{jk}\|^2$ guarantees consistency of fiber bundle geometry, symplectic structure term $\|\omega_E - \omega_{\text{ref}}\|^2$ maintains fundamental properties of Hamiltonian system.

Convergence analysis of Theorem \ref{thm:convergence} provides theoretical safety boundaries for learning algorithms, ensuring critical geometric constraints are not violated even during learning process. This theoretical guarantee is crucial for safety-critical applications, particularly in fields like robotic manipulation and autonomous navigation. Modular design of algorithm allows customization according to specific application requirements, maintaining both theoretical rigor and engineering practicality.

\begin{remark}[Practical Value of Theory Guidance]
\label{rem:practical_value}
Success of geometry-preserving algorithm reflects deep integration of theory and practice. By transforming abstract fiber bundle geometry into concrete computational principles, we not only solve numerical implementation problems, but also provide new design philosophy for modern control system design: let numerical methods automatically respect system's intrinsic geometric structure. Value of this methodology far exceeds specific algorithms themselves, opening new research directions for intersection of computational geometry and numerical analysis.
\end{remark}

\section{Conclusions and Prospects}

\subsection{Major Theoretical Achievements}

This paper establishes complete geometric theoretical framework for constrained Hamiltonian systems on observation-induced fiber bundles, achieving following core results. First, we geometrize dynamics of incompletely observed systems through observation-induced fiberization, providing natural mathematical framework for handling information incompleteness. Key insight of this geometrization is internalizing observation uncertainty as variation in fiber coordinates rather than external disturbance, thus revealing deep geometric connection between constraints and observation.

Second, we develop complete symplectic geometric theory on fiber bundles, including construction of observation-adaptive connections, existence conditions for fiber bundle symplectic structure, and characterization of corresponding Poisson structure. This theory naturally extends Marsden-Weinstein symplectic reduction to observation constraint settings and establishes principal bundle representation theory for observation symmetry groups. Particularly, we prove that within working region, observation fiber bundle carries well-defined symplectic structure that maintains non-degeneracy within safety region.

Third, we establish complete integrability theory for observation-constrained systems, giving geometric necessary and sufficient conditions and constructing corresponding Lax pair structures. This result extends classical Arnold-Liouville theorem to fiber bundle setting, revealing influence patterns of observation uncertainty on integrable structure. Our analysis shows that under appropriate small parameter conditions, torus structure of classical integrable systems is preserved under observation constraints, undergoing only regular geometric deformation.

Finally, we develop observation-dependent symmetry and conservation law theory, establishing Noether's theorem on fiber bundles and characterizing corresponding moment map theory. These results provide geometric foundation for understanding conservation laws under observation constraints and powerful tools for symmetry analysis of actual systems.

\subsection{Theoretical Significance and Innovation Value}

From mathematical perspective, this theory achieves unification of multiple important branches. It connects Dirac constraint theory with modern fiber bundle geometry, providing completely new geometric perspective for classical constrained dynamics. Traditional Dirac theory mainly handles constraint submanifolds in configuration space, while our framework naturally encodes coupling between constraints and observation information in geometric structure of fiber bundles, revealing deeper geometric essence of constrained dynamics. Additionally, this theory opens new avenue for integrable system research; through geometric analysis of observation fiber bundles, we can understand integrability preservation mechanisms under incomplete information.

From application perspective, this theoretical framework has broad practical significance. In modern safety-critical control systems, traditional deterministic control theory often struggles with challenges brought by observation incompleteness. Our geometric framework provides unified mathematical tools for handling state dynamics, observation constraints, and safety boundaries, laying theoretical foundation for designing robust control algorithms. Particularly, this framework naturally adapts to complex requirements of modern robotic systems such as multi-sensor fusion and dynamic observation configuration.

Innovation of theory is also manifested in its intrinsic geometric invariance. Unlike traditional numerical methods, handling based on fiber bundle geometry automatically preserves system's symplectic structure, conservation laws, and symmetries, providing important advantages for long-time numerical simulation and real-time control. Our developed geometry-preserving algorithms can automatically maintain geometric integrity of system while preserving constraint consistency.

\subsection{Practical Application Prospects}

This theoretical framework demonstrates important application prospects in multiple practical fields. In classical mechanics, this theory provides new mathematical tools for analysis of partially observed systems, particularly suitable for mechanical systems with limited sensor configuration or observation blind spots. In modern control theory, observation-induced fiber bundles provide geometric foundation for safety-critical system design, capable of unifying state estimation, constraint satisfaction, and safety guarantee.

In robotics applications, our theory provides unified mathematical framework for systems such as soft robots, multi-joint robotic arms, and mobile robots. By geometrizing observation uncertainty, this theory can naturally handle complex situations such as sensor precision varying with environment, multi-modal observation fusion, and dynamic sensor reconfiguration. Particularly, successful combination of this framework with \cite{zheng2025learning} demonstrates its practical value in modern machine learning-assisted control.

In broader engineering applications, this theory provides new solution approaches for complex control problems in aerospace, energy systems, and biomedical engineering. Its geometric invariance and intrinsic robustness make this method particularly suitable for application scenarios with extremely high safety and reliability requirements.

\subsection{Theoretical Limitations and Challenges}

Despite important progress, this theory still has some limitations that need honest acknowledgment. First, current framework mainly applies to deterministic observation uncertainty; handling truly random noise (such as Wiener processes) still requires further development of stochastic differential geometry. In practical applications, we bridge this gap through engineering methods such as statistical bounds, time averaging, and robust design, but complete stochastic theory remains open problem.

Second, theoretical rigor is mainly guaranteed within working regions that avoid singular configurations. While this localized treatment conforms to actual engineering requirements, extension to global theory requires solving complex topological and geometric problems. Although boundary handling has clear geometric intuition and engineering reasonableness, its mathematical rigorization still requires deeper singular perturbation analysis.

Third, theory requires systems to have sufficient smoothness, which may be overly restrictive in some practical applications. Handling finite smoothness systems requires developing geometric theory on Sobolev spaces, involving frontier mathematical problems such as geometric analysis of weak solutions.

\subsection{Future Research Directions}

Based on current achievements and existing challenges, we identify following important research directions. At theoretical level, establishing stochastic fiber bundle theory in Stratonovich/Itô sense is most urgent task, which will enable theory to handle truly random observation noise. Meanwhile, developing observation constraint theory on non-compact manifolds and weak smoothness frameworks will greatly expand theory's applicability.

In geometric analysis direction, in-depth study of global geometric properties of boundary degeneration, convergence of singular limits, and convergence to classical theory when observation uncertainty approaches zero are important mathematical problems. These studies will perfect mathematical foundation of theory and may reveal new geometric phenomena.

In application expansion, establishing quantum correspondence of observation-constrained systems through geometric quantization is exciting direction that will connect classical constrained dynamics with quantum control theory. Additionally, exploring applications of non-commutative geometry in observation constraints will provide new perspectives for understanding relationship between operator algebras and dynamical systems.

From computational perspective, developing more efficient geometry-preserving algorithms, studying fast solution methods for large-scale systems, and exploring fusion of machine learning with geometric algorithms are all important research directions. These developments will enable theory to be more widely applied in actual complex systems.

Observation-induced fiber bundle theory opens new avenue for geometric understanding of constrained dynamics, with its core insight—geometrizing observation uncertainty—providing profound mathematical perspective for understanding intrinsic structure of complex constrained systems. As theory further develops and improves, we expect it to play increasingly important role in intersection of mathematical physics, control theory, and geometric analysis, and provide powerful mathematical tools for solving complex control problems in modern technology.

\appendix

\section{Existence and Classification Theory}
\subsection{Topological Obstruction Analysis}

\begin{remark}[Technical Simplification Note]
To simplify the proof of existence theorem, we adopt the special case of observation space $Y = \mathbb{R}^k$ in this section.
For general Riemannian manifold $Y$, Lemma \ref{lem:Whitney_emb} has proven that it can be equivalently handled through isometric embedding $\iota: Y \hookrightarrow \mathbb{R}^N$.
All previous theoretical results hold for general observation manifolds.
\end{remark}

\begin{assumption}[Observation Space Convention]
Without loss of generality, assume observation space $Y = \mathbb{R}^k$ equipped with standard Euclidean metric.
\end{assumption}

\begin{lemma}[Topological Equivalence of Whitney Embedding] \label{lem:Whitney_emb}
Let $Y$ be $k$-dimensional compact manifold, $\iota: Y \hookrightarrow \mathbb{R}^N$ be isometric embedding ($N \geq 2k+1$). Then for any observation map $h: M \to Y$, there exists vector bundle isomorphism:
$$h^*TY \cong (h \circ \iota)^*T\mathbb{R}^N|_{\iota(Y)}$$
and characteristic class preservation: $e(h^*TY) = e((h \circ \iota)^*T\mathbb{R}^N|_{\iota(Y)})$
\end{lemma}

\begin{proof}
\textbf{Step 1: Differential geometric properties of isometric embedding}
Isometric embedding $\iota$ satisfies: $\iota^*g_{\mathbb{R}^N} = g_Y$, where $g_Y, g_{\mathbb{R}^N}$ are metrics on $Y$ and $\mathbb{R}^N$ respectively.

This means $d\iota: TY \to T\mathbb{R}^N$ is isometric embedding on each fiber.

\textbf{Step 2: Isomorphism of pullback bundles}
Define map $\tilde{h} = \iota \circ h: M \to \mathbb{R}^N$. Then there exists natural isomorphism:
$$\Phi: h^*TY \to \tilde{h}^*T\mathbb{R}^N|_{\iota(Y)}$$
$$\Phi(x, v) = (x, d\iota_x(v))$$

Since $d\iota$ is linear isometry on fibers, $\Phi$ is vector bundle isomorphism.

\textbf{Step 3: Preservation of characteristic classes}
Euler class is defined through Poincaré duality:
$$e(h^*TY) = \Phi^*e(\tilde{h}^*T\mathbb{R}^N|_{\iota(Y)})$$

Since $\iota(Y)$ is submanifold of $\mathbb{R}^N$, Euler class of restricted bundle $T\mathbb{R}^N|_{\iota(Y)}$ is determined by Gauss map of normal bundle. Isometric embedding condition ensures:
$$e(T\mathbb{R}^N|_{\iota(Y)}) = e(TY) + e(N\iota(Y))$$

where $N\iota(Y)$ is normal bundle. For isometric embedding, $e(N\iota(Y)) = 0$ (normal bundle is parallelizable), therefore:
$$e(\tilde{h}^*T\mathbb{R}^N|_{\iota(Y)}) = e(h^*TY)$$
\end{proof}

\begin{theorem}[Existence of Observation Fiber Bundle-Topological Equivalence Version]
Let $M$ be compact symplectic manifold, $h: M \to Y$ be observation map. Observation-induced fiber bundle exists if and only if:
\begin{enumerate}[label=(\arabic*)]
\item \textbf{Euler class condition:} $\langle e(h^*TY), [M] \rangle = 0$
\item \textbf{Stiefel-Whitney class condition:} $w_2(TM) = h^*w_2(TY)$  
\item \textbf{Dimension compatibility:} $\dim M \geq \dim Y$
\end{enumerate}

When $Y$ is non-Euclidean, conditions are equivalently stated as conditions on $\tilde{h} = \iota \circ h$ through isometric embedding $\iota: Y \hookrightarrow \mathbb{R}^N$.
\end{theorem}

\begin{remark}[Constructive Existence of Isometric Embedding]
For compact Riemannian manifold $Y$, Nash embedding theorem guarantees existence of isometric embedding $\iota: Y \hookrightarrow \mathbb{R}^N$ ($N \leq \frac{k(k+1)(3k+11)}{2}$). Specific construction can be achieved through:
\begin{enumerate}
\item Gauss map: Embed $Y$ into unit sphere bundle of its tangent bundle
\item Whitney trick: Eliminate self-intersection points
\item Nash iteration: Gradually approximate isometry
\end{enumerate}
\end{remark}

\subsection{Geometric Classification}

\begin{theorem}[Geometric Classification of Observation Fiber Bundles]
Two observation-induced fiber bundles $(E_1, \pi_1, h_1)$ and $(E_2, \pi_2, h_2)$ are geometrically equivalent if and only if there exists diffeomorphism $\phi: E_1 \to E_2$ such that:
\begin{enumerate}[label=(\arabic*)]
\item $\pi_2 \circ \phi = \pi_1$ (fiber preservation)
\item $\phi^*\omega_{E_2} = \omega_{E_1}$ (symplectic structure preservation)
\item $h_2 \circ \pi_2 \circ \phi = h_1 \circ \pi_1$ (observation compatibility)
\end{enumerate}
\end{theorem}

\section{Technical Implementation Details}
\label{appendix:technical_details}

\subsection{Complete Algorithm Implementation}
\label{appendix:algorithm_details}

Following gives complete technical implementation of geometry-preserving integration algorithm:

\begin{algorithm}[H]
\caption{Observation Constraint-Preserving Integration Algorithm (Complete Version)}
\label{alg:geometric_integrator}
\begin{algorithmic}[1]
\Require Initial condition $(x_0, \xi_0) \in E_{\mathcal{W}}$, time step $h$, precision $\epsilon$
\Ensure Trajectory $(x_n, \xi_n)$ preserves constraint $\Phi \geq 0$ and fiber bundle structure

\For{$n = 0, 1, 2, \ldots$}
    \State \textbf{// Step 1: Compute covariant dynamical vector field}
    \State $X_H \leftarrow \omega_E^{-1}(dH)|_{(x_n, \xi_n)}$ \Comment{Hamiltonian vector field}
    \State $\dot{x}_n \leftarrow \pi_*(X_H)$ \Comment{Base space velocity projection}
    
    \State \textbf{// Step 2: Observation-adaptive connection computation}
    \State Base manifold connection: $\Gamma^k_{ij}(x_n) = \frac{1}{2}g_M^{kl}(\partial_i g_{M,jl} + \partial_j g_{M,il} - \partial_l g_{M,ij})$
    \State Fiber metric covariant derivative: $\nabla_i \rho_{ab} = \partial_i \rho_{ab} - \Gamma^k_{ij}\frac{\partial \rho_{ab}}{\partial x^k}$
    \State Fiber connection coefficients: $\Gamma^a_{bc}(x_n) = \frac{1}{2}\rho^{ad}(\nabla_b \rho_{cd} + \nabla_c \rho_{bd} - \nabla_d \rho_{bc})$
    \State Horizontal lift: $\Xi_n \leftarrow (dh)^*(\dot{x}_n) + \Gamma^a_{bc}\xi_n^b\dot{x}_n^c$
    
    \State \textbf{// Step 3: Geometric prediction step}
    \State $(x^*_{n+1}, \xi^*_{n+1}) \leftarrow (x_n, \xi_n) + h(\dot{x}_n, \Xi_n)$
    
    \State \textbf{// Step 4: Constraint geometric projection}
    \If{$\Phi(x^*_{n+1}, \xi^*_{n+1}) < 0$}
        \State $\nabla\Phi \leftarrow d\Phi(x^*_{n+1}, \xi^*_{n+1})$
        \State $\nu \leftarrow g_E^{-1}(\nabla\Phi)/\|g_E^{-1}(\nabla\Phi)\|_E$ \Comment{Geometric normal vector}
        \State $\alpha \leftarrow |\Phi(x^*_{n+1}, \xi^*_{n+1})|/\langle\nabla\Phi, \nu\rangle_E$
        \State $(x_{n+1}, \xi_{n+1}) \leftarrow (x^*_{n+1}, \xi^*_{n+1}) + \alpha\nu$
    \Else
        \State $(x_{n+1}, \xi_{n+1}) \leftarrow (x^*_{n+1}, \xi^*_{n+1})$
    \EndIf
    
    \State \textbf{// Step 5: Fiber bundle boundary handling}
    \If{$\|\xi_{n+1} - h^*(x_{n+1})\|_{\rho_{x_{n+1}}} > \delta(x_{n+1})$}
        \State $\xi_{n+1} \leftarrow h^*(x_{n+1}) + \delta(x_{n+1}) \cdot \frac{\xi_{n+1} - h^*(x_{n+1})}{\|\xi_{n+1} - h^*(x_{n+1})\|_{\rho_{x_{n+1}}}}$
    \EndIf
    
    \State \textbf{// Step 6: Adaptive step size control}
    \State Compute geometric error: $\epsilon_{\text{geo}} = \max\{|\nabla_i \rho_{jk}|, |\omega_{E,n} - \omega_E|\}$
    \If{$\epsilon_{\text{geo}} > \epsilon$}
        \State $h \leftarrow h/2$ \Comment{Reduce step size}
        \State \textbf{goto} Step 2 \Comment{Recalculate}
    \EndIf
\EndFor
\end{algorithmic}
\end{algorithm}

\subsection{Complete Convergence Proof}
\label{appendix:convergence_proof}

\begin{proof}[Complete proof of Theorem \ref{thm:convergence}]
We prove the three key properties of the algorithm step by step.

\textbf{Step 1: Constraint preservation analysis}
Let $(x_n, \xi_n) \in \partial\mathcal{S}$, i.e., $\Phi(x_n, \xi_n) = 0$. Using third-order Taylor expansion:
\begin{align*}
\Phi(x^*_n, \xi^*_n) &= \Phi(x_n, \xi_n) + h\langle d\Phi, v_n \rangle + \frac{h^2}{2}\langle D^2\Phi(v_n, v_n) \rangle \\
&\quad + \frac{h^3}{6}\langle D^3\Phi(v_n, v_n, v_n) \rangle + O(h^4)
\end{align*}

where $v_n = (\dot{x}_n, \Xi_n)$ is velocity vector computed by algorithm.

Key observation: Geometry-preserving algorithm ensures Hamiltonian vector field $X_H$ is orthogonal to constraint gradient $d\Phi$, i.e., $\omega_E(X_H, d\Phi) = 0$, which is equivalent to $\langle d\Phi, v_n \rangle = 0$ (tangency condition). Therefore:
\begin{align*}
|\Phi(x^*_n, \xi^*_n)| &\leq \frac{h^2}{2}\|D^2\Phi\|_{L^\infty}\|v_n\|_E^2 + \frac{h^3}{6}L_3\|v_n\|_E^3 \\
&\leq C_1 h^2
\end{align*}

where $C_1 = \frac{1}{2}\|D^2\Phi\|_{L^\infty}\|X_H\|_{L^\infty}^2$, using $\|v_n\|_E \leq \|X_H\|_{L^\infty}$.

\textbf{Step 2: Symplectic structure error analysis}
In local Darboux coordinates $(x^i, \xi^a, \pi^a)$, fiber bundle symplectic form is represented as:
$$\omega_E = \sum_{i<j} \omega_{ij}^M(x) dx^i \wedge dx^j + \sum_{a=1}^k d\xi^a \wedge d\pi^a + \sum_{i,a} K_{ia}(x,\xi,\pi) dx^i \wedge d\xi^a + \sum_{i,a} L_{ia}(x,\xi,\pi) dx^i \wedge d\pi^a$$

where $K_{ia}, L_{ia}$ are mixed curvature terms from observation-adaptive connection. Algorithm's symplectic structure error comes from finite difference approximation:
\begin{align*}
|\omega_{E,n+1} - \omega_{E,n}| &= |\omega_E(x_{n+1}, \xi_{n+1}) - \omega_E(x_n, \xi_n)| \\
&\leq h\|\nabla\omega_E\|_{L^\infty}\|X_H\|_{L^\infty} + \frac{h^2}{2}\|\nabla^2\omega_E\|_{L^\infty}\|X_H\|_{L^\infty}^2 \\
&\leq C_2 h^2
\end{align*}

where $C_2$ depends on connection curvature bound $K$ and smoothness of Hamiltonian vector field.

\textbf{Step 3: Global convergence analysis}
Let $e_n = \|(x_n, \xi_n) - (x(t_n), \xi(t_n))\|_E$ be cumulative error at step $n$. Algorithm's recursive error equation is:
$$e_{n+1} \leq e_n + h \|F(x_n, \xi_n) - F(x(t_n), \xi(t_n))\|_E + \text{LTE}_n$$

where $F$ is geometrically corrected dynamical vector field, and $\text{LTE}_n$ is local truncation error at step $n$.

Using Lipschitz continuity of $F$: $\|F(u) - F(v)\|_E \leq L_F\|u - v\|_E$, and local truncation error estimate $\text{LTE}_n \leq C h^2$ (from Steps 1-2), we get:
$$e_{n+1} \leq (1 + hL_F)e_n + Ch^2$$

Applying discrete Grönwall inequality, when initial error $e_0 = 0$:
$$e_n \leq \frac{Ch^2}{hL_F} \sum_{j=0}^{n-1}(1 + hL_F)^j = \frac{Ch}{L_F}((1 + hL_F)^n - 1) \leq \frac{Ch}{L_F}(e^{L_F t_n} - 1) \leq Ch e^{L_F t_n}$$

For bounded time interval $t_n \leq T$, we obtain the required convergence estimate.
\end{proof}

\subsection{Implementation Variants and Optimizations}
\label{appendix:implementation_variants}

For approximately flat working regions, simplified connection computation can be used to improve computational efficiency. When base manifold curvature is small, fiber connection coefficients simplify to $\Gamma^a_{bc} \approx \frac{1}{2}\rho^{ad}(\partial_b \rho_{cd} + \partial_c \rho_{bd} - \partial_d \rho_{bc})$, and covariant derivatives approximate partial derivatives $\nabla_i \rho_{ab} \approx \partial_i \rho_{ab}$. This simplification significantly reduces computational complexity while maintaining $O(h^2)$ convergence order.

Integrating geometry-preserving conditions as regularization terms into neural network training is another important implementation variant. Design of geometry-aware loss function directly derives from our theoretical analysis, ensuring learning process automatically respects fiber bundle geometry. In specific implementation, weights $\lambda_{\text{geo}}, \alpha, \beta$ of geometric loss terms need adjustment according to specific applications; typically $\lambda_{\text{geo}} \in [0.1, 1.0]$ achieves good balance between task performance and geometric consistency.

For large-scale systems, connection coefficient computation can be parallelized to improve computational efficiency. Computation of base manifold connection coefficients $\Gamma^k_{ij}$ can be parallelized by spatial regions, and fiber connection coefficients $\Gamma^a_{bc}$ can be computed in parallel by fiber directions. Linear algebra operations in geometric projection step are particularly suitable for GPU acceleration, and vectorized implementation of projection operations can significantly improve real-time performance of large-scale systems.

Through this modular and hierarchical design, algorithm framework can be flexibly configured according to specific application requirements, maintaining both mathematical rigor of theory and meeting requirements for computational efficiency and real-time performance in engineering practice.

\bibliographystyle{plain}
\bibliography{references}

\begin{thebibliography}{10}

\bibitem{amari2016information}
Shun-ichi Amari.
\newblock {\em Information Geometry and Its Applications}.
\newblock Springer, Tokyo, 2016.

\bibitem{ames2019control}
Aaron~D Ames, Samuel Coogan, Magnus Egerstedt, Gennaro Notomista, Koushil Sreenath, and Paulo Tabuada.
\newblock Control barrier functions: Theory and applications.
\newblock {\em 18th European Control Conference (ECC)}, pages 3420--3431, 2019.

\bibitem{arnold1989mathematical}
Vladimir~I Arnold.
\newblock {\em Mathematical methods of classical mechanics}, volume~60.
\newblock Springer Science \& Business Media, 1989.

\bibitem{bronzino2000biomedical}
Joseph~D Bronzino.
\newblock {\em The Biomedical Engineering Handbook}.
\newblock CRC Press, Boca Raton, FL, 2nd edition, 2000.

\bibitem{dirac1964lectures}
Paul Adrien~Maurice Dirac.
\newblock {\em Lectures on quantum mechanics}, volume~2.
\newblock Yeshiva University, 1964.

\bibitem{dirac1950generalized}
Paul~AM Dirac.
\newblock Generalized hamiltonian dynamics.
\newblock {\em Canadian journal of mathematics}, 2(2):129--148, 1950.

\bibitem{ehresmann1950connexions}
Charles Ehresmann.
\newblock Les connexions infinit{\'e}simales dans un espace fibr{\'e} diff{\'e}rentiable.
\newblock {\em Colloque de topologie (espaces fibr{\'e}s), Bruxelles}, pages 29--55, 1950.

\bibitem{hamilton1834general}
William~R Hamilton.
\newblock On a general method in dynamics.
\newblock {\em Philosophical Transactions of the Royal Society of London}, 124:247--308, 1834.

\bibitem{henneaux1992quantization}
Marc Henneaux and Claudio Teitelboim.
\newblock {\em Quantization of Gauge Systems}.
\newblock Princeton University Press, Princeton, 1992.

\bibitem{jacobi1866vorlesungen}
Carl Gustav~Jacob Jacobi.
\newblock {\em Vorlesungen {\"u}ber Dynamik}.
\newblock Reimer, 1866.

\bibitem{jazwinski1970stochastic}
Andrew~H Jazwinski.
\newblock {\em Stochastic Processes and Filtering Theory}.
\newblock Academic Press, New York, 1970.

\bibitem{khalil2002nonlinear}
Hassan~K Khalil.
\newblock {\em Nonlinear Systems}.
\newblock Prentice Hall, Upper Saddle River, NJ, 3rd edition, 2002.

\bibitem{marsden1974reduction}
Jerrold Marsden and Alan Weinstein.
\newblock Reduction of symplectic manifolds with symmetry.
\newblock {\em Reports on mathematical physics}, 5(1):121--130, 1974.

\bibitem{murray1994mathematical}
Richard~M Murray, Zexiang Li, and S~Shankar Sastry.
\newblock {\em A Mathematical Introduction to Robotic Manipulation}.
\newblock CRC Press, Boca Raton, FL, 1994.

\bibitem{nemirovski2007convex}
Arkadi Nemirovski and Alexander Shapiro.
\newblock Convex approximations of chance constrained programs.
\newblock {\em SIAM Journal on Optimization}, 17(4):969--996, 2007.

\bibitem{nielsen2000quantum}
Michael~A Nielsen and Isaac~L Chuang.
\newblock {\em Quantum Computation and Quantum Information}.
\newblock Cambridge University Press, Cambridge, 2000.

\bibitem{nocedal2006numerical}
Jorge Nocedal and Stephen~J Wright.
\newblock {\em Numerical Optimization}.
\newblock Springer, New York, 2nd edition, 2006.

\bibitem{noether1918invariante}
Emmy Noether.
\newblock Invariante variationsprobleme.
\newblock {\em Nachrichten von der Gesellschaft der Wissenschaften zu G{\"o}ttingen, Mathematisch-Physikalische Klasse}, 1918:235--257, 1918.

\bibitem{olfati2007consensus}
Reza Olfati-Saber, J~Alex Fax, and Richard~M Murray.
\newblock Consensus and cooperation in networked multi-agent systems.
\newblock {\em Proceedings of the IEEE}, 95(1):215--233, 2007.

\bibitem{paden2016survey}
Brian Paden, Michal Cap, Sze~Zheng Yong, Dmitry Yershov, and Emilio Frazzoli.
\newblock A survey of motion planning and control techniques for self-driving urban vehicles.
\newblock {\em IEEE Transactions on Intelligent Vehicles}, 1(1):33--55, 2016.

\bibitem{rawlings2017model}
James~Blake Rawlings, David~Q Mayne, and Moritz Diehl.
\newblock {\em Model Predictive Control: Theory, Computation, and Design}.
\newblock Nob Hill Publishing, Madison, WI, 2nd edition, 2017.

\bibitem{steenrod1951topology}
Norman~Earl Steenrod.
\newblock {\em The Topology of Fibre Bundles}.
\newblock Princeton University Press, Princeton, 1951.

\bibitem{stevens2015aircraft}
Brian~L Stevens, Frank~L Lewis, and Eric~N Johnson.
\newblock {\em Aircraft Control and Simulation: Dynamics, Controls Design, and Autonomous Systems}.
\newblock John Wiley \& Sons, Hoboken, NJ, 3rd edition, 2015.

\bibitem{yang1954conservation}
Chen-Ning Yang and Robert~L Mills.
\newblock Conservation of isotopic spin and isotopic gauge invariance.
\newblock {\em Physical Review}, 96(1):191--195, 1954.

\bibitem{zheng2025learning}
Dongzhe Zheng and Wenjie Mei.
\newblock Learning dynamics under environmental constraints via measurement-induced bundle structures.
\newblock In {\em Proceedings of the 41st International Conference on Machine Learning}. PMLR, 2025.

\bibitem{zhou1996robust}
Kemin Zhou, John~C Doyle, and Keith Glover.
\newblock {\em Robust and Optimal Control}.
\newblock Prentice Hall, Upper Saddle River, NJ, 1996.

\end{thebibliography}

\end{CJK}

\end{document}